\documentclass[english,11pt]{article}%
\usepackage[T1]{fontenc}
\usepackage[latin9]{inputenc}
\usepackage{amsmath}
\usepackage{amssymb}
\usepackage{amsfonts}
\usepackage{babel}
\usepackage{fancyhdr}
\usepackage{color}
\usepackage[toc]{appendix}
\usepackage{graphicx}%
\providecommand{\U}[1]{\protect\rule{.1in}{.1in}}
\makeatletter
\newtheorem{theorem}{Theorem}

\newtheorem{condition}[theorem]{Condition}

\newtheorem{definition}[theorem]{Definition}
\newtheorem{example}{Example}

\newtheorem{lemma}[theorem]{Lemma}

\newtheorem{proposition}[theorem]{Proposition}
\newtheorem{remark}[theorem]{Remark}

\numberwithin{equation}{section}
\numberwithin{example}{section}
\numberwithin{theorem}{section}

\newenvironment{proof}[1][Proof]{\noindent\textbf{#1.} }{\ \rule{0.5em}{0.5em}}
\makeatother
\begin{document}

\title{Formulation and properties of a divergence used to compare probability measures without absolute continuity}
\author{Paul Dupuis\thanks{%
Research supported in part by the National Science Foundation (NSF-DMS-1904992).} \  and Yixiang Mao\thanks{%
Research supported in part by the Air Force Office of Scientific Research (FA-9550-18-1-0214).} }
\maketitle

\begin{abstract}
\noindent This paper develops a new divergence that generalizes relative entropy and can be used to compare probability measures 
without a requirement of absolute continuity.
We establish properties of the divergence, and in particular derive and exploit a representation
as an infimum convolution of optimal transport cost and relative entropy. 
Also included are examples of computation and approximation of the divergence,
and the demonstration of properties that are useful when one quantifies model uncertainty.
\end{abstract}


\section{Introduction}

To compare different probabilistic models for a given application, one
needs a notion of \textquotedblleft distance\textquotedblright\ between
the distributions. The specification of this distance is a subtle
issue. Probability models are typically large or infinite dimensional, and the
usefulness of the distance will depend on its mathematical properties. Is it
convenient for analysis and optimization? Does it scale well with system size?

For situations that require an analysis of (probabilistic) model form uncertainly, the
quantity known as relative entropy (or Kullback-Leibler divergence)
is the most widely used such distance. This is true because relative entropy
has all the attractive properties asked for in the last paragraph, and
many more. (Relative entropy is not a true metric since it is not symmetric in
its arguments, but owing to its other attributes it is more widely
used for these purposes than any legitimate metric.)

The definition of relative entropy is as follows. Suppose $S$ is a Polish
space with metric $d(\cdot,\cdot)$ and associated Borel $\sigma$-algebra
$\mathcal{B}$. Let $\mathcal{P}(S)$ be the space of probability measures over
$(S,\mathcal{B})$. If $\mu,\nu\in\mathcal{P}(S)$ and $\mu$ is absolutely continuous with respect to $\nu$ (denoted $\mu
\ll\nu$), then
\[
R(\mu\lVert\nu)\doteq\int_{S}\left(  \log\frac{d\mu}{d\nu}\right)  d\mu
\]
(even though $\log{d\mu}/{d\nu}$ can take both positive and negative values,
as we discuss in the beginning of section \ref{section_2}, the definition is never ambiguous). Otherwise, we set $R(\mu\lVert\nu)=\infty$. 

While we cannot go into
all the reasons why relative entropy is so useful, it is essential that we
describe why it is convenient for the analysis of 
model form uncertainty. This is due to a
dual pair of variational formulas which relate $R(\mu\lVert\nu)$, integrals
with respect $\mu$, and what are called \textbf{risk-sensitive} integrals with respect
to $\nu$. Let $M_{b}(S)$ denote the set of bounded and measurable functions on $S$. Then
\cite[Proposition 1.4.2 and Lemma 1.4.3]{dupell4} gives

\begin{equation}
R(\mu\left\Vert \nu\right.  )=\sup_{g\in M_{b}(S)}\left\{  \int_{S}gd\mu
-\log\int_{S}e^{g}d\nu\right\},   \label{eqn:var_forms}%
\end{equation}
and for any $g\in M_{b}(S)$,
\[
\log\int_{S}e^{g}d\nu=\sup_{\mu
\in\mathcal{P}(S)}\left\{  \int_{S}gd\mu-R(\mu\left\Vert \nu\right.
)\right\}  .
\]
It is immediate from either of these that for $\mu,\nu\in\mathcal{P}(S)$
and $g\in M_{b}(S)$,%
\[
\int_{S}gd\mu\leq R(\mu\left\Vert \nu\right.  )+\log\int_{S}e^{g}d\nu.
\]
If we interpret $\nu$ as the \textbf{nominal}\ or \textbf{design }model
(chosen perhaps on the basis of data or for computational tractability) and
$\mu$ as the \textbf{true} model (or at least a {more accurate} model),
then according to the last display one obtains a bound on an integral with respect to the true model. 
(In fact by introducing a parameter one can obtain bounds that are in some sense optimal\cite{dupkatpanple}.) 
We typically interpret the integral $\int_{S}gd\mu$ as a \textbf{performance measure}, and so we have a bound
on the performance of the system under the true distribution in terms of the
relative entropy distance $R(\mu\left\Vert \nu\right.  )$, plus a
risk-sensitive performance measure under the design model. From this elementary
but fundamental inequality, and by exploiting the helpful qualitative and
quantitative properties of relative entropy, there has emerged a set of tools
that can be used to answer many questions where probabilistic model form
uncertainty is important, including 
\cite{baylov, brecsi,chodup,dupjampet,dupkatpanple,glaxu,hansar,lam,limshawat, arnlau, petjamdup}.

However, relative entropy has one important shortcoming: for the bound to be
meaningful we must have $R(\mu\left\Vert \nu\right.  )<\infty$, which imposes
the requirement of absolute continuity of the true model with respect to the
design model. For various uses, such as model building and model
simplification, this restriction can be significant. In the context of model
building, it can happen that one attempts to fit distributions to data by
comparing an empirical measure constructed using data with the elements of a
parameterized family, such as a collection of Gaussian distributions. In this
case the two distributions one would compare are singular, and so relative
entropy cannot be used. A second example, and one that occurs frequently in
the physical sciences, operations research and elsewhere,
is that a detailed model (such as the population process of a
chemical reaction network, which takes values in a lattice) is approximated by
a simpler process that takes values in the continuum (for example a diffusion
process). For exactly the same reason as in the previous example, these
processes, as well as their corresponding stationary distributions, are not
absolutely continuous.

Because relative entropy is not directly applicable to such problems,
significant effort has been put into investigating alternatives (\cite{blakanmur,blamur} and references therein). 
A class that has attracted some
attention (e.g., in the machine learning community) are the \textit{Wasserstein} or, more generally, 
\textit{optimal transport} distances \cite{kolparthosleroh,racrus,vil}. These distances, which are true
metrics, have certain attractive properties but also some 
shortcomings. The most important shortcomings are: (a) Wasserstein distances
do not in general scale well with respect to system dimension, and (b) such 
distances do not have an interpretation as the dual of a strictly convex function. To
be a little more concrete about point (b), it is
the strict concavity of the mapping
\[
g\rightarrow\int_{S}gd\mu-\log\int_{S}e^{g}d\nu
\]
in the variational representation for $R(\mu\left\Vert \nu\right.  )$ that
leads to tight bounds when applied to problems of control or optimization of
stochastic uncertain systems. In contrast, the analogous variational representation for
Wasserstein type distances involves the mapping $g\rightarrow\int_{S}%
gd\mu-\int_{S}gd\nu$. Point (a) is an issue in applications
to problems from the physical sciences, where large time horizons and large
dimensions are common.

Rather than give up entirely the attractive features of the dual pair
$(R(\mu\left\Vert \nu\right.  ),\log\int_{S}e^{g}d\nu)$, an alternative is to
be more restrictive regarding the class of costs or performance measures for
which bounds are required.  Indeed, the requirement of absolute
continuity in relative entropy is entirely due to the very large class of
functions, $M_{b}(S)$, appearing in (\ref{eqn:var_forms}). For a collection
$\Gamma\subset M_{b}(S)$ one can consider in lieu of $R(\mu\left\Vert
\nu\right.  )$ what we call the $\Gamma$-divergence, which is defined by
\[
G_{\Gamma}(\mu\left\Vert \nu\right.  )\doteq \sup_{g\in\Gamma}\left\{  \int_{S}%
gd\mu-\log\int_{S}e^{g}d\nu\right\}  .
\]
By imposing regularity conditions on $\Gamma$ (e.g., Lipschitz continuity,
additional smoothness) one generates (under mild additional conditions on $\Gamma$)
divergences which relax the absolute continuity condition. Thus one is trading
restrictions on the class of performance measures or observables for which
bounds are valid, for the enlargement of the class of distributions to which
the bounds apply. These divergences are of course not as nice as relative
entropy, but one can prove that they retain versions of its most important
properties. In addition, the dual function (which serves as the cost to be
minimized when considering problems of optimization or control) remains
$\log\int_{S}e^{g}d\nu$. This is important because the corresponding risk-sensitive
optimization and optimal control problems are well studied in the literature.

In our formulation of the $\Gamma$-divergence the underlying idea is that to extend the range of probability measures that can be compared, one must restrict the class of integrands that will be considered. 
However, this leads directly to an interesting connection with the Wasserstein distance
mentioned previously, which is that for suitable collections $\Gamma$ we will 
prove the inf-convolution expression
\begin{align*}
G_{\Gamma}(\mu\left\Vert \nu\right.  )&=\inf_{\gamma\in\mathcal{P}(S)}\left\{
W_{\Gamma}(\mu-\gamma)+R(\gamma\left\Vert \nu\right.  )\right\}  ,
\end{align*}
where $W_{\Gamma}$ is the Wasserstein metric whose dual (sup) formulation uses
the set of functions $\Gamma$. Moreover one recovers relative entropy by taking the
limit $b\rightarrow\infty$ in $G_{b\Gamma}(\mu\left\Vert \nu\right.  )$, which
may be useful if one wants to allow relatively small violations of the absolute
continuity restriction, while at the same time taking advantage of simple
approximations for the Wasserstein distance in the high transportation cost limit.

The organization of the paper is as follows.
In Section 2 we define the $\Gamma$-divergence, and prove the first main result of this paper, which is the inf-convolution formula described above (Theorem \ref{thm:main}). In Section 3, we show several properties of the
$\Gamma$-divergence, and establish a convex duality formula for the $\Gamma$-divergence. 
Section 4 investigates the $\Gamma$-divergence for a special choices of $\Gamma$, which are sets of bounded Lipschitz continuous functions. 
We establish a relation between $\Gamma$-divergence and optimal transport cost, and prove existence and uniqueness for optimizers of variational representations of $\Gamma$-divergence (Theorem \ref{optimizer}), and also a formula for directional derivatives of the $\Gamma$-divergence (Theorem \ref{first_variation}). 
Section 5 considers limits for the $\Gamma$-divergence, and in Section 6 there is a preliminary discussion on how one can apply the $\Gamma$-divergence to obtain uncertainty quantification bounds.  

As a last remark we note that the paper \cite{amakaroiz} defines a ``relaxation'' of
Wasserstein distance by putting in an entropy term of the mass-transfer
matrix. The new divergence so defined is easier to compute than the original
Wasserstein distance, but is not the same as the divergences we develop here.

\section{Definition of the ${\Gamma}$-divergence}\label{section_2}

Throughout this section, $S$ is a Polish space with metric $d(\cdot,\cdot)$
and associated Borel $\sigma$-alegra $\mathcal{B}$. $C_{b}(S)$ denotes the
space of all bounded continuous functions from $S$ to $\mathbb{R}$, and
$M_{b}(S)$ denotes the space of all bounded measurable functions from $S$ to
$\mathbb{R}$. Let $\mathcal{P}(S)$ be the space of probability measures over
$(S,\mathcal{B})$, $\mathcal{M}(S)$ be the space of finite signed (Borel)
measures over $(S,\mathcal{B})$, and $\mathcal{M}_{0}(S)$ be the subspace of
$\mathcal{M}(S)$ whose total mass is $0$. $\overline{\mathbb{R}}%
\doteq\mathbb{R}\cup\{\infty\}$ is the extended real numbers. Throughout this
section, we consider $C_{b}(S)$ equipped with weak topology induced by
$\mathcal{M}(S)$. Thus for $f_n, f \in C_b(S)$, $f_{n}\rightarrow f$ if $\int_{S}f_{n}d\mu
\rightarrow\int_{S}fd\mu$ for all $\mu\in\mathcal{M}(S)$.

We recall that 
\[
R(\mu\lVert\nu)\doteq\int_{S}\left(  \log\frac{d\mu}{d\nu}\right)  d\mu
\]
whenever $\mu$ is absolutely continuous with respect to $\nu$. For $t\in\mathbb{R}$ define $t^{-}%
\doteq-(t\wedge0)$. Since the function $s(\log s)^{-}$ is bounded for
$s\in\lbrack0,\infty)$, whenever $\mu\ll\nu$,%

\[
\int_{S} \left(  \log\frac{d\mu}{d\nu}\right)  ^{-}d\mu=\int_{S} \frac{d\mu
}{d\nu}\left(  \log\frac{d\mu}{d\nu}\right)  ^{-}d\nu<\infty.
\]
Thus $R(\mu\lVert\nu)$ is always well defined.

We recall the Donsker-Varadhan variational representation (\ref{eqn:var_forms})
for relative
entropy.
We will use equation (\ref{eqn:var_forms}) as an equivalent
characterization of $R(\cdot\lVert\nu)$ on $\mathcal{P}(S)$, and consider an extension to
$\mathcal{M}(S)$ in the following lemma. With an abuse of notation, we will also call the extended function $R$.  
To set up the functionals of interest on a space with the proper structure (locally convex Hausdorff space),
we will use that 
\begin{equation}
\sup_{g\in C_{b}(S)}\left\{  \int_{S}gd\mu-\log\int_{S}e^{g}d\nu\right\}
=\sup_{g\in M_{b}(S)}\left\{  \int_{S}gd\mu-\log\int_{S}e^{g}d\nu\right\}
\label{eqn:altchar}%
\end{equation}
\cite[Lemma 1.4.3(a)]{dupell4}(It is worth notating in this reference, (\ref{eqn:altchar}) is only proved for $\mu,\nu\in \mathcal{P}(S)$. However, the exact same argument applies for $\mu,\nu\in\mathcal{M}(S)$, and we are using the latter version here). The fact that one obtains the same value
when supremizing over the smaller class $C_{b}(S)$ is closely related to the fact that 
$R(\mu\lVert\nu)$ is finite only when $\mu\ll\nu$.

\begin{lemma}
\label{lem:RE}Consider $R:\mathcal{M}(S)\times\mathcal{P}(S)\rightarrow
(-\infty,\infty]$ defined by (\ref{eqn:var_forms}). Then

\begin{enumerate}
\item $R(\mu\lVert\nu)\geq0$ and $R(\mu\lVert\nu)=0$ if and only if $\mu=\nu$,

\item $R(\cdot\lVert\cdot)$ is convex,

\item $R(\mu\lVert\nu)=\infty$ if $\mu\in\mathcal{M}(S)\backslash
\mathcal{P}(S)$.
\end{enumerate}
\end{lemma}

\begin{proof}
If we prove item 3, then items 1 and 2 will follow from the corresponding
statements when $\mu$ is restricted to $\mathcal{P}(S)~$\cite{dupell4}. If
$m=\mu(S)\neq1$, then taking $g(x)\equiv c$ a constant,
\[
\int_{S}gd\mu-\log\int_{S}e^{g}d\nu=c\mu(S)-c=c(m-1).
\]
Since $m\neq1$ and $c\in\mathbb{R}$, the right hand side of equation
(\ref{eqn:var_forms}) is $\infty$.

Suppose next that $\mu(S)=1$ but $\mu\in\mathcal{M}(S)\backslash
\mathcal{P}(S)$. Then there exist sets $A,B\in\mathcal{B}$ such that $A\cap
B=\varnothing,A\cup B=S,\mu(A)<0$ and $\mu(B)>0$. For $c>0$, let $g(x)=-c$ for
$x\in A$ and $g(x)=0$ for $x\in B$. Then
\[
\int_{S}gd\mu-\log\int_{S}e^{g}d\nu=c\left\vert \mu(A)\right\vert -C_{c},
\]
where $C_{c}\in(\log\nu(B),0)$ for all $c$. Letting $c\rightarrow\infty$ and
using (\ref{eqn:altchar}) shows $R(\mu\lVert\nu)=\infty$.
\end{proof}

\vspace{.5\baselineskip}
Though relative entropy has very attractive regularity and optimization
properties, as noted $R(\mu\lVert\nu)$ is finite if and only if $\mu\ll\nu$.
As such, it cannot be used to give a meaningful notion of ``distance'' without this absolute continuity restriction. In
order to define a meaningful divergence for a pair of probability measures
that are not mutually absolute continuous, but at the same time not to lose the
useful properties of the \textquotedblleft dual\textquotedblright\ function
$g\rightarrow\log\int_{S}e^{g}d\nu$ appearing in (\ref{eqn:var_forms}), a natural
approach is to restrict the set of test functions in the variational formula.
We define a criterion for the classes of \textquotedblleft
admissible\textquotedblright\ test functions we want to use.

\begin{definition}
\label{access} Let $\Gamma$ be a subset of $C_{b}(S)$ endowed with the inherited weak topology. We call $\Gamma$
\textbf{admissible} if the following hold.

1) $\Gamma$ is convex and closed.

2) $\Gamma$ is symmetric in that $g\in\Gamma$ implies $-g\in\Gamma$, and
$\Gamma$ contains all constant functions.

3) $\Gamma$ is determining for $\mathcal{P}(S)$, i.e., for any $\mu,\nu
\in\mathcal{P}(S)$ with $\mu\neq\nu$, there exists $g\in\Gamma$ such that%
\[
\int_{S} gd\mu\neq\int_{S} gd\nu.
\]

\end{definition}

We next define a new divergence by restricting the class of test functions in
the definition of relative entropy.

\begin{definition}
\label{def:defofV}Fix $\nu\in \mathcal{P}(S)$. For $\mu\in\mathcal{M}(S)$, we define the
$\mathbf{\Gamma}$\textbf{-divergence} associated with the admissible set $\Gamma$ by%
\[
G_{\Gamma}(\mu\lVert\nu)\doteq\sup_{g\in\Gamma}\left\{  \int_{S} gd\mu-\log\int_{S}
e^{g}d\nu\right\}  .
\]
We also define the following related quantity. For $\eta\in\mathcal{M}(S)$ let%
\[
W_{\Gamma}(\eta)\doteq\sup_{g\in\Gamma}\left\{  \int_{S} gd\eta\right\}
=\sup_{g\in\mathcal{C}_{b}(S)}\left\{  \int_{S} gd\eta-\infty1_{\{g\in\Gamma^{c}%
\}}\right\}  .
\]

\end{definition}

When $\Gamma$ is clear based on context, we will drop the subscript from
$G_{\Gamma}$ and $W_{\Gamma}$. Using a similar argument as in Lemma
\ref{lem:RE}, one can show that $G_{\Gamma}(\mu\lVert\nu)=\infty$ if $\mu(S)\neq 1$. The next theorem
states an important property of the ${\Gamma}$-divergence, which is that it
can be written as a convolution involving relative entropy and $W_{\Gamma}$.

\begin{theorem}
\label{thm:main} Assume $\Gamma$ is an admissible set. Then for $\mu\in \mathcal{M}(S)$, $\nu\in\mathcal{P}(S)$,
\[
G_{\Gamma}(\mu\lVert\nu)=\inf_{\gamma\in\mathcal{P}(S)}\left\{  R(\gamma
\lVert\nu)+W_{\Gamma}(\mu-\gamma)\right\}
\]

\end{theorem}

\begin{remark}
The theorem tells us that by restricting the set of test functions in the
variational representation of relative entropy, we get a quantity which is an
inf-convolution of relative entropy and a metric. It will be pointed out in
Section \ref{Wass} that by restricting $\Gamma$ to Lipschitz functions
with respect to a cost function $c(x,y)$ that satisfies some specified
conditions, $W_{\Gamma}(\mu-\nu)$ will be the corresponding optimal transport
cost from $\mu$ to $\nu$.
\end{remark}

The rest of this section is focused on the proof of Theorem \ref{thm:main}. In
order to do this, we need a few definitions and also will find it convenient
to consider a more general setting.

\begin{definition}
Points $x$ and $y$ in a topological space $Y$ can be \textbf{separated} if
there exists an open neighborhood $U$ of $x$ and an open neighborhood $V$ of
$y$ such that $U$ and $V$ are disjoint ($U\cap V=\varnothing$). $Y$ is a
$\mathbf{Hausdorff}$ space if all distinct points in $Y$ are pairwise separable.
\end{definition}

\begin{definition}
A subset $C$ of a topological vector space $Y$ over the number field
$\mathbb{R}$ is

1. $\mathbf{convex}$ if for any $x,y\in C$ and any $t\in\lbrack0,1]$,
$tx+(1-t)y\in C$,

2. $\mathbf{balanced}$ if for all $x\in C$ and any $\lambda\in\mathbb{R}$ with
$|\lambda|\leq1$, $\lambda x\in C$,

3. $\mathbf{absorbant}$ if for all $y\in Y$, there exists $t>0$ and $x\in C$
such that $y=tx$.

A topological vector space $Y$ is called $\mathbf{locally}$ $\mathbf{convex}$
if the origin has a local topological basis of convex, balanced and absorbent sets.
\end{definition}

\begin{definition}
For a topological vector space $Y$ over the number field $\mathbb{R}$, its
$\mathbf{topological\ dual\ space}$ $Y^{*}$ is defined as the space of all
continuous linear functionals ${\displaystyle \varphi:Y\to{\mathbb{R} }}$.

The $\mathbf{weak^{\ast}\ topology}$ on $Y^{\ast}$ is the topology induced by
$Y$. In other words, it is the coarsest topology such that functional
$y:Y^{\ast}\rightarrow\mathbb{R}$, $y(\varphi)=\varphi(y)$ is continuous in
$Y^{\ast}$.

For $y\in Y$ and $\varphi\in Y^{\ast}$, we also write $\langle y,\varphi
\rangle\doteq\varphi(y)=y(\varphi)$.
\end{definition}

Now let $Y$ be a Hausdorff locally convex space with $Y^{\ast}$ being its
topological dual space and endowed with the weak* topology.

\begin{definition}
For a function $f:Y\rightarrow\overline{\mathbb{R}}$, its $\mathbf{convex}%
\ \mathbf{dual}$ $f^{\ast}:Y^{\ast}\rightarrow\overline{\mathbb{R}}$ is
defined by%
\[
f^{\ast}(z)=\sup_{y\in Y}\left\{  \langle y,z\rangle-f(y)\right\}  .
\]

\end{definition}

\begin{definition}
Let $f_{1},f_{2}:Y\rightarrow\overline{\mathbb{R}}$ be two functions. We
define the inf-convolution of $f_{1}$ and $f_{2}$ by%
\[
\left[  f_{1}\Box f_{2}\right]  (y)\doteq\inf_{y_{1}\in Y}\{f_{1}(y_{1}%
)+f_{2}(y-y_{1})\}.
\]

\end{definition}

\begin{definition}
For a function $f:Y\rightarrow\overline{\mathbb{R}}$ the
$\mathbf{lower\ semicontinuous\ hull}$ $\overline{f}$ is defined by%
\[
\overline{f}(x)\doteq\sup\{g(x):g\leq f,g:Y\rightarrow\overline{\mathbb{R}%
}\ is\ continuous\}.
\]

\end{definition}

\begin{definition}
A convex function $f:Y\rightarrow\overline{\mathbb{R}}$ is \textbf{proper} if
there exists $y\in Y$ such that $f(y)<\infty$. The \textbf{domain} of a
convex, proper funciton $f$ is defined by%
\[
\mathrm{dom}(f)\doteq\{y\in Y:f(y)<\infty\}.
\]

\end{definition}

Now let us introduce an important lemma. 

\begin{lemma}
\label{inf-cov} \cite[Theorem 2.3.10]{botgrawan} Let $f_{i}:Y\rightarrow
\overline{\mathbb{R}}$ be convex, proper and lower-semicontinuous functions
fulfilling $\bigcap_{i=1}^{m}\mathrm{dom}(f_{i})\neq\varnothing$. Then one has%
\[
\left(  \sum_{i=1}^{m}f_{i}\right)  ^{\ast}=\overline{f_{1}^{\ast}\Box
\cdots\Box f_{m}^{\ast}}.
\]

\end{lemma}

In our use we take $Y=C_{b}(S)$ equipped with topology induced by
$\mathcal{M}(S)$, i.e., the topological basis around $g\in Y$ is taken as sets
of the form%

\[
\left\{  f\in Y:\int_{S} fd\mu_{k}\in\left(  \int_{S}gd\mu_{k}-\epsilon_{k}%
,\int_{S}gd\mu_{k}+\epsilon_{k}\right)  ,k=1,2,\dots,m\right\}  ,
\]
where $m\in\mathbb{N},\{\mu_{k}\}_{k=1,2,\dots,m}\subset\mathcal{M}(S)$ and
$\epsilon_{k}>0,k=1,2,\dots,m$ are arbitrary. It can be easily verified that
under this topology, $C_{b}(S)$ is a Hausdorff locally convex space, with
$C_{b}(S)^{\ast}=\mathcal{M}(S)$ \cite[Theorem 3.10]{rud}. 
For $g\in C_{b}(S)$ and $\mu\in
\mathcal{M}(S)$, we define the bilinear form%

\[
\langle g,\mu\rangle\doteq\int_{S}gd\mu.
\]

We are now ready to prove the main theorem.

\vspace{.5\baselineskip}
\begin{proof}
[Proof of Theorem \ref{thm:main}]Define $H_{1},H_{2}:C_{b}(S)\rightarrow
\overline{\mathbb{R}}$ by%
\[
H_{1}(g)\doteq\log\int_{S}e^{g}d\nu\text{ and }H_{2}(g)\doteq\infty
1_{\Gamma^{c}}(g).
\]
Then
\begin{align*}
G_{\Gamma}(\mu\lVert\nu) &  =\sup_{g\in\Gamma}\left\{  \int_{S}gd\mu-\log
\int_{S}e^{g}d\nu\right\}  \\
&  =\sup_{g\in C_{b}(S)}\left\{  \int_{S}gd\mu-\log\int_{S}e^{g}d\nu
-\infty1_{\Gamma^{c}}(g)\right\}  \\
&  =\left(  H_{1}+H_{2}\right)  ^{\ast}(\mu).
\end{align*}

Notice that $\{0\}\in\mathrm{dom}(H_{1})\cap\mathrm{dom}(H_{2})\neq
\varnothing$, and both $H_{1}$ and $H_{2}$ are proper and convex. For
lower-semicontinuity, under the topology induced by $\mathcal{M}(S)$, $H_{1}$ is
actually continuous, and $H_{2}$ is lower semicontinuous since $\Gamma$ is
closed. Thus,
by Lemma \ref{inf-cov}%
\[
G_{\Gamma}(\mu\lVert\nu)=(H_{1}+H_{2})^{\ast}(\mu)=[\overline{H_{1}^{\ast}\Box
H_{2}^{\ast}}](\mu).
\]
By equation (\ref{eqn:var_forms}) and the definition of $W_{\Gamma}$, we know that%
\[
R(\mu\lVert\nu)=H_{1}^{\ast}(\mu)\text{ and }W_{\Gamma}(\eta)=H_{2}^{\ast
}(\eta).
\]
In the following display, the first equality is due to the definition of
inf-convolution, and the second is since $R(\gamma\lVert\nu)<\infty$ only when
$\gamma\in\mathcal{P}(S)$:%
\begin{align*}
H_{1}^{\ast}\Box H_{2}^{\ast}(\mu)  &  =\inf_{\gamma\in\mathcal{M}(S)}\left\{
R(\gamma\lVert\nu)+W_{\Gamma}(\mu-\gamma)\right\} \\
&  =\inf_{\gamma\in\mathcal{P}(S)}\left\{  R(\gamma\lVert\nu)+W_{\Gamma}%
(\mu-\gamma)\right\}  .
\end{align*}

Thus the last thing we need to prove is that $H_{1}^{\ast}\Box H_{2}^{\ast}$
is lower semicontinuous. Note that relative entropy is lower semicontinuous in
the first argument in the weak topology \cite[Lemma 1.4.3 (b)]{dupell4}, and
$W_{\Gamma}$ is lower semicontinuous in the weak topology since
it is the supremum of a collection of linear functionals. Let%
\[
F(\mu)\doteq H_{1}^{\ast}\Box H_{2}^{\ast}(\mu)=\inf_{\gamma\in\mathcal{P}%
(S)}\left\{  R(\gamma\lVert\nu)+W_{\Gamma}(\mu-\gamma)\right\}  .
\]
\noindent Consider any sequence $\mu_{n}\Rightarrow\mu$ with $\mu_n,\mu\in\mathcal{M}(S)$. Here ``$\Rightarrow$'' means convergence in the weak$^*$ topology, i.e., for any $f\in C_b(S)$, $\int f d\mu_n\to\int f d\mu$. Let $\varepsilon>0$, and
for each $\mu_n$ let $\gamma_n$ satisfy%
\[
R(\gamma_{n}\lVert\nu)+W_{\Gamma}(\mu_{n}-\gamma_{n})\leq F(\mu_{n}%
)+\varepsilon.
\]
\noindent We want to show that%
\begin{equation}
\liminf_{n\rightarrow\infty}F(\mu_{n})\geq F(\mu). \label{eqn:lsc}%
\end{equation}
If $\liminf_{n\rightarrow\infty} F(\mu_n)=\infty$, the inequality above holds
automatically. Assuming $\liminf_{n\rightarrow\infty} F(\mu_n)<\infty$, let
$n_k$ be a subsequence such that%

\[
\lim_{k\rightarrow\infty}F(\mu_{n_{k}})=\liminf_{n\rightarrow\infty}F(\mu
_{n}).
\]
Notice that%
\[
R(\gamma_{n_{k}}\lVert\nu)\leq R(\gamma_{n_{k}}\lVert\nu)+W_{\Gamma}%
(\mu_{n_{k}}-\gamma_{n_{k}})\leq F(\mu_{n_{k}})+\varepsilon.
\]
Since $\{F(\mu_{n_{k}})\}_{k\geq1}$ is bounded, we know that $\{\gamma_{n_{k}%
}\}_{k\geq1}$ is tight \cite[Lemma 1.4.3(c)]{dupell4}. Then we can take a further
subsequence that converges weakly. For simplicity of notation, let $n_{k}$
denote this subsequence, and let $\gamma_{\infty}$ denote the weak limit of
$\gamma_{n_{k}}$. Then using the lower semicontinuity of $R(\cdot\lVert\nu)$
on $\mathcal{P}(S)$ and the lower semicontinuity of $W_{\Gamma}$ on
$\mathcal{M}(S)$,%
\begin{align*}
\liminf_{n\rightarrow\infty}F(\mu_{n})+\varepsilon &  =\lim_{k\rightarrow
\infty}F(\mu_{n_{k}})+\varepsilon\\
&  \geq\lim_{k\rightarrow\infty}\left[  R(\gamma_{n_{k}}\lVert\nu
)+W(\mu_{n_{k}}-\gamma_{n_{k}})\right] \\
&  \geq R(\gamma_{\infty}\lVert\nu)+W(\mu-\gamma_{\infty})\\
&  \geq\inf_{\gamma\in\mathcal{P}(S)}\left\{  R(\gamma\lVert\nu)+W_{\Gamma
}(\mu-\gamma)\right\} \\
&  =F(\mu).
\end{align*}
Since $\varepsilon>0$ is arbitrary this establishes (\ref{eqn:lsc}), and thus
$F$ is lower semicontinuous in $\mathcal{M}(S)$. The theorem is proved.
\end{proof}

\section{Properties of the $\Gamma$-divergence}

Theorem \ref{thm:main} provides an interesting characterization of the $\Gamma$-divergence. Before we
continue to specific choices of $\Gamma$, we first state some general
properties associated with $\Gamma$-divergence. Throughout this section we fix an
admissible set $\Gamma$, and thus drop the subscript from $G_{\Gamma}$ and
$W_{\Gamma}$ in this section. Also, now that we have established the
expression for $G$ as an inf-convolution as in Theorem \ref{thm:main}, we no
longer need to consider $G$ as a function on $\mathcal{M}(S)\times
\mathcal{P}(S)$, and instead can consider it just on $\mathcal{P}(S)\times
\mathcal{P}(S)$, since we want to use $G$ as a measure of how two probability distributions differ.

\begin{lemma}
\label{basic} For $(\mu,\nu)\in\mathcal{P}(S)\times\mathcal{P}(S)$ define
$G(\mu\lVert\nu)$ by Definition \ref{def:defofV} and assume $\Gamma$ is
admissible. Then the following properties hold.

1) $G(\mu\lVert\nu)\geq0$, with $G(\mu\lVert\nu)=0$ if and only if $\mu=\nu$.

2) $G(\mu\lVert\nu)$ is a convex and lower semicontinuous function of
$(\mu,\nu)$. In particular, $G(\mu\lVert\nu)$ is a convex, lower
semicontinuous function of each variable $\mu$ or $\nu$ separately.

3) $G(\mu\lVert\nu)\leq R(\mu\lVert\nu)$ and $G(\mu\lVert\nu)\leq W(\mu-\nu)$.
\end{lemma}

\begin{remark}
1) The first property justifies our calling $G$ a divergence as the term is used in information theory.

2) Relative entropy has the property that for each fixed $\nu\in\mathcal{P}(S)$,
$R(\cdot\lVert\nu)$ is strictly convex on $\{\mu\in\mathcal{P}(S):R(\mu
\lVert\nu)<\infty\}$. However,
$G(\cdot\lVert\nu)$ in general is not strictly convex.
\end{remark}

\begin{proof}
[Proof of Lemma \ref{basic}]
1) As noted in Lemma \ref{lem:RE}, $R(\cdot\lVert\cdot)$ is non-negative
\cite[Lemma 1.4.1]{dupell4}, and for any $\mu\in\mathcal{P}(S)$%
\[
W(\mu)=\sup_{g\in\Gamma}\left\{  \int_{S}gd\mu\right\}  \geq\int_{S}0d\mu=0.
\]
\noindent Thus%
\[
G(\mu\lVert\nu)=\inf\{R(\mu_{1}\lVert\nu)+W(\mu_{2}):\mu_{1}+\mu_{2}=\mu
\}\geq0.
\]
Also by Lemma \ref{lem:RE}, $R(\mu_{1}\lVert\nu)=0$ if and only if
$\mu_{1}=\nu$. Thus $G(\mu\lVert\nu)=0$ if and only if%
\[
W(\mu-\nu)=\sup_{g\in\Gamma}\left\{  \int_{S}gd(\mu-\nu)\right\}  =0,
\]
which tells us $\mu=\nu$ since $\Gamma$ is admissible.

2) This is a straightforward corollary of Theorem \ref{thm:main}, since the
supremum of a collection of linear and continuous functionals is both convex
and lower semicontinuous.

3) This follows from Theorem \ref{thm:main} and that $R(\nu\lVert\nu)=W(0)=0$.
\end{proof}

\vspace{.5\baselineskip}
For relative entropy we have the following lemma \cite[Proposition 1.4.2]{dupell4}.

\begin{lemma}\label{REvar} For all $g\in C_{b}(S)$%
\[
\log\int_{S}e^{g}d\nu=\sup_{\mu\in\mathcal{P}(S)}\left\{  \int_{S}gd\mu
-R(\mu\lVert\nu)\right\}  ,
\]
where the supremum is achieved uniquely at $\mu_{0}$ satisfying%
\[
\frac{d\mu_{0}}{d\nu}(x)\doteq\frac{e^{g(x)}}{\int_{S}e^{g}d\nu}.
\]
\end{lemma}

A similar duality formula holds for the $\Gamma$-divergence when $g\in\Gamma$.

\begin{theorem}
\label{incredible} If $\Gamma$ is admissible then for $g\in\Gamma$%
\[
\log\int_{S} e^{g}d\nu=\sup_{\mu\in\mathcal{P}(S)}\left\{  \int_{S} gd\mu-G(\mu
\lVert\nu)\right\}  .
\]
\end{theorem}

\begin{proof}
Using the definition of $G$ divergence 
\begin{align*}
\sup_{\mu\in\mathcal{P}(S)}\left\{  \int_{S}gd\mu-G(\mu\lVert\nu)\right\}   &
=\sup_{\mu\in\mathcal{P}(S)}\left\{  \int_{S}gd\mu-\sup_{f\in\Gamma}\left\{
\int_{S}fd\mu-\log\int_{S}e^{f}d\nu\right\}  \right\} \\
&  \leq\sup_{\mu\in\mathcal{P}(S)}\left\{  \int_{S}gd\mu-\left\{  \int
_{S}gd\mu-\log\int_{S}e^{g}d\nu\right\}  \right\} \\
&  =\log\int_{S}e^{g}d\nu.
\end{align*}
On the other hand, we know for relative entropy that%
\[
\log\int_{S}e^{g}d\nu=\sup_{\mu\ll\nu}\left\{  \int_{S}gd\mu-R(\mu\lVert
\nu)\right\}  .
\]
Since $G(\mu\lVert\nu)\leq R(\mu\lVert\nu)$,
\begin{align*}
\log\int_{S}e^{g}d\nu &  =\sup_{\mu\ll\nu}\left\{  \int_{S}gd\mu-R(\mu
\lVert\nu)\right\} \\
&  \leq\sup_{\mu\ll\nu}\left\{  \int_{S}gd\mu-G(\mu\lVert\nu)\right\} \\
&  \leq\sup_{\mu\in\mathcal{P}(S)}\left\{  \int_{S}gd\mu-G(\mu\lVert
\nu)\right\}.
\end{align*}
The statement of the theorem follows from the two inequalities.
\end{proof}

\vspace{.5\baselineskip}
The last theorem has two important implications.
The first is related to the fact that Lemma \ref{REvar} implies bounds for $\int_S g d\mu$ when $R(\mu\lVert\nu)$ is bounded,
and observation that has served as the basis for the analysis of various aspects of model form uncertainty \cite{chodup,dupkatpanple}. Using Theorem \ref{incredible}, we obtain analogous bounds on $\int_S g d\mu$ for $g\in\Gamma$ when  $G(\mu\lVert\nu)$ is bounded. Applications of these bounds will be further developed elsewhere.
The second is that for $g\in\Gamma$, if we take $\mu_{0}$ as defined in Lemma \ref{REvar}, then
\begin{align*}
\log\int_{S} e^{g} d\nu &  = \int_{S} g d\mu_{0} - R(\mu_{0}\lVert\nu)\\
&  \leq\int_{S} g d\mu_{0} - G(\mu_{0}\lVert\nu)\\
&  \leq\sup_{\mu\in \mathcal{P}(S)}\left\{  \int_{S} g d\mu- G(\mu\lVert\nu) \right\} \\
&  = \log\int_{S} e^{g} d\nu,
\end{align*}
where the first inequality comes from $G(\mu_{0}\lVert\nu)\leq
R(\mu_0\lVert\nu)$. Since both
inequalities above must be equalities, we must have%
\[
R(\mu_{0}\lVert\nu) = G(\mu_{0}\lVert\nu).
\]
The next lemma gives a more detailed picture of $G(\mu\lVert\nu)$ when $\mu\ll\nu$.

\begin{lemma}
\label{goodcase}
For $\mu,\nu\in \mathcal{P}(S)$, if $\mu\ll\nu$ then 
$$G(\mu\lVert\nu) = \sup_{\gamma\in\mathcal{A}(S)}\left\{\int_S \log\left(\frac{d\gamma}{d\nu}\right)d\mu\right\},$$
where 
\[
\mathcal{A}(S)\doteq\left\{\gamma\in\mathcal{P}(S): \gamma\ll\nu,\exists g \in \Gamma \mbox{ such that }\frac{d\gamma}{d\nu}(x) =e^{g(x)} \mbox{ for } x\in\mathrm{supp}(\nu) \right\}.
\]
\end{lemma}

\begin{proof}
We use the definition $$G(\mu\lVert \nu) = \sup_{g\in\Gamma} \left\{\int_S g d\mu - \log \int_S e^g d\nu \right\}$$
to prove this lemma.
For any $g\in\Gamma$, we define $\gamma_g \in \mathcal{P}(S)$ by the relation
$$\frac{d\gamma_g}{d\nu}(x)=\frac{e^{g(x)}}{\int_S e^g d\nu}$$
for $x\in \mathrm{supp}(\nu)$, and $\gamma_g(\mathrm{supp}(\nu)^c)=0$. Then for $x\in \mathrm{supp}(\nu)$,
$$\log\left(\frac{d\gamma_g}{d\nu}(x)\right) = g(x) - \log\int_S e^g d\nu.$$
Since $\mu\ll\nu$, we have
$$\int_S \log\left(\frac{d\gamma_g}{d\nu}\right)d\mu = \int_S gd\mu - \log\int_S e^g d\nu,$$
and thus
$$G(\mu\lVert \nu) = \sup_{g\in\Gamma} \left\{\int_S g d\mu - \log \int_S e^g d\nu \right\}\leq \sup_{\gamma\in\mathcal{A}(S)}\left\{\int_S \log\left(\frac{d\gamma}{d\nu}\right)d\mu\right\}.$$

On the other hand, for any $\gamma\in \mathcal{A}(S)$, by definition, we can find a $g_\gamma\in \Gamma$ such that
$$g_\gamma(x) = \log\left(\frac{d\gamma}{d\nu}(x)\right)$$
for $x\in\mathrm{supp}(\nu)$. Then
$$\int_S g_\gamma d\mu - \log\int_S e^{g_\gamma} d\nu = \int_S\log\left(\frac{d\gamma}{d\nu}\right)d\mu.$$
Thus
$$\sup_{\gamma\in \mathcal{A}(S)}\left\{\int_S \log\left(\frac{d\gamma}{d\nu}\right)d\mu\right\}\leq \sup_{g\in\Gamma} \left\{\int_S g d\mu - \log \int_S e^g d\nu \right\}=G(\mu\lVert \nu). $$
Combining  the two inequalities completes the proof.
\end{proof}

\begin{remark} 
When $\mu\in\mathcal{A}(S)$ we always have $G(\mu\lVert\nu)=R(\mu\lVert\nu)$.
This is because if $\gamma\in \mathcal{A}(S)$ then  $\mu\ll\gamma$, and therefore
$$\int_S\log\left(\frac{d\mu}{d\nu}\right)d\mu - \int_S\log\left(\frac{d\gamma}{d\nu}\right)d\mu =\int_S\log\left(\frac{d\mu}{d\gamma}\right)d\mu = R(\mu\lVert\gamma)\geq 0. $$
Rearranging gives 
\[
\int_S \log\left(\frac{d\gamma}{d\nu}\right)d\mu=R(\mu\lVert\nu)-R(\mu\lVert\gamma), 
\]
and so 
$$G(\mu\lVert\nu)= \sup_{\gamma\in\mathcal{A}(S)}\left\{\int_S \log\left(\frac{d\gamma}{d\nu}\right)d\mu\right\}= R(\mu\lVert\nu).$$

This statement is not valid when $\mu\ll\nu$ does not hold, since then $\log(d\gamma/d\nu)$ is not defined in $\mathrm{supp}(\mu)\backslash\mathrm{supp}(\nu)$, thus 
$$\int_S \log\left(\frac{d\gamma}{d\nu}\right)d\mu$$
is not well defined.
\end{remark}

\section{Connection with optimal transport theory}\label{OT}

\label{Wass}

In the proceeding sections, we discussed general properties for the $\Gamma$-divergence 
with an admissible set $\Gamma\subset C_{b}(S)$. In this section, we
discuss specific choices of $\Gamma$ which relate the $\Gamma$-divergence with
optimal transport theory. First we state some well known results in
optimal transport theory.

\subsection{Preliminary results from optimal transport theory}

The results in this section are from \cite[Chapter 4]{racrus}. The
general Monge-Kantorovich mass transfer problem with given marginals $\mu
,\nu\in\mathcal{P}(S)$ and cost function $c:S\times S\rightarrow\mathbb{R}%
_{+}$ is%
\[
\mathcal{C}(c;\mu,\nu)\doteq\inf_{\pi\in\Pi(\mu,\nu)}\left\{  \int_{S\times
S}c(x,y)\pi(dx,dy)\right\}  ,
\]
where $\Pi(\mu,\nu)$ denotes the collection of all probability
measures on $S\times S$ with first and second marginals being $\mu$ and $\nu$, respectively.

A natural dual problem with respect to this is%
\[
\mathcal{B}(c;\rho)\doteq \sup_{f\in\mathrm{Lip}(c,S;C_{b}(S))}\left\{  \int
_{S}f(x)\rho(dx)\right\}  ,
\]
where $\rho=\mu-\nu$, $C_b(S)$ denotes the set of bounded continuous functions mapping $S$ to $\mathbb{R}$ and%

\begin{align}\label{Lip_bdd}
\mathrm{Lip}(c,S;C_{b}(S))\doteq\left\{  f\in C_{b}(S):f(x)-f(y)\leq
c(x,y)\mbox{ for all } x,y\in S\right\}  .
\end{align}
We want to know when%
\begin{equation}
\mathcal{C}(c;\mu,\nu)=\mathcal{B}(c,\rho) \label{eqn:duality}%
\end{equation}
holds. The following is a necessary and sufficient condition.
As with many results in this section,
one can extend in a trivial way to the case where costs are bounded from below, rather than non-negative.
Recall that $S$ is a Polish space.

\begin{condition}
\label{con:crep}
There is a nonempty
subset $Q\subset C_{b}(S)$ such that the cost
$c:S\times S\rightarrow[0,\infty]$  has the representation%
\begin{equation}
c(x,y)=\sup_{u\in Q}\left(  u(x)-u(y)\right)  \quad\text{for\ all}\ (x,y)\in
S\times S.\label{eqn:crep}
\end{equation}
\end{condition}

\begin{theorem}
\cite[Theorem 4.6.6]{racrus}\label{massdual} Under Condition \ref{con:crep},  
 (\ref{eqn:duality}) holds.
\end{theorem}

\begin{remark}
Condition \ref{con:crep} implies that $c$ satisfies the triangle inequality, i.e., for all $x,y,z\in S$
$$c(x,z) \leq c(x,y)+c(y,z).$$
This follows easily from 
\begin{align*}
\sup_{u\in Q}\left(  u(x)-u(z)\right) &= \sup_{u\in Q}\left(  (u(x)-u(y)) +(u(y)-u(z))\right)\\
&\leq \sup_{u\in Q}\left(  u(x)-u(y)\right) +\sup_{u\in Q}\left(  u(y)-u(z)\right).
\end{align*}

On the other hand, Condition \ref{con:crep} also allows for a wide range of choices of
$c(x,y)$. For example, suppose that $c$ is a continuous metric on $S$, where
continuity is with respect to the underlying metric
of $S$. Then we can choose%
\[
Q=\left\{  \min(c(x,x_{0}),n):x_{0}\in S,n\in\mathbb{N}\right\}  .
\]
It is easily verified that $Q\subset C_{b}(S)$, and that with this choice of
$Q$ \eqref{eqn:crep} holds.
\end{remark}

\subsection{$\Gamma$-divergence with the choice $\Gamma=\mathrm{Lip}(c,S;C_{b}%
(S))$}

Suppose $\Gamma=\mathrm{Lip}(c,S;C_{b}(S))$, with $c:S\times S\rightarrow
[0, \infty]$ satisfying 
 Condition \ref{con:crep}. To make the presentation simple, we have assumed that $c$ is non-negative,
 and further assume it is symmetric, meaning $c(x,y) = c(y,x) \geq 0$ for any $x,y\in S$.
To distinguish from $W_\Gamma(\mu-\nu)$ for general $\Gamma$,
we denote the transport cost for $\mu, \nu\in \mathcal{P}(S)$ by
\[
W_{c}(\mu,\nu)\doteq\sup_{g\in \mathrm{Lip}(c,S;C_{b}(S))}\left\{  \int_{S}gd(\mu-\nu)\right\} .
\]
Then by Theorem \ref{massdual}%
\[
W_{c}(\mu,\nu)=\sup_{g\in\mathrm{Lip}(c,S;C_{b}(S))}\left\{  \int_{S}gd(\mu-\nu)\right\}
=\inf_{\pi\in\Pi(\mu,\nu)}\left\{  \int_{S\times S}c(x,y)\pi(dx,dy)\right\}
.
\]

\begin{condition}\label{mea-det}
Suppose $\mathrm{Lip}(c,S;C_{b}(S))$ is measure determining, i.e., 
for all $\mu,\nu\in\mathcal{P}(S)$, $\mu \neq \nu$, there exists $f\in \mathrm{Lip}(c,S;C_{b}(S))$ such that 
$$\int_S fd\mu \neq \int_S f d\nu.$$
\end{condition}
Under Condition \ref{mea-det}, $\Gamma$ is admissible (see Definition \ref{access}), and by
Theorem \ref{thm:main}%
\begin{equation}
G_{\Gamma}(\mu\lVert\nu)=\sup_{g\in\Gamma}\left\{  \int_{S}gd\mu-\log\int
_{S}e^{g}d\nu\right\}  =\inf_{\gamma\in\mathcal{P}(S)}\left\{  W_{c}(\mu,\gamma)+R(\gamma
\lVert\nu)\right\}  . \label{Vari}%
\end{equation}
Hence by choosing $\Gamma$ properly, we get that the $\Gamma$-divergence is an
infimal convolution of relative entropy, which is a convex function of likelihood ratios, and an
optimal transport cost, which depends on a cost structure on the space $S$.
Natural questions to raise here are the following. \quad

i) Do there exist optimizers $\gamma^{\ast}$ and $g^{\ast}$ in the variational
problem (\ref{Vari})? If so, are they unique?

ii) How can one characterize $\gamma^{\ast}$ and $g^{\ast}$?

iii) For a fixed $\nu\in\mathcal{P}(S)$, what is the effect of a perturbation of
$\mu$ on $G_{\Gamma}(\mu\lVert\nu)$?

\medskip We will address these questions sequentially in this section. From now on, we will drop the subscript $\Gamma$ in this section for the simplicity of writing. We consider the case where $G(\mu\lVert\nu)<\infty$. To
impose additional constraints on $\mu$ and $\nu$ such that $G(\mu\lVert\nu)<\infty$ holds, we make a further
assumption on $c$.

\begin{condition}
\label{finite} There exist $a:S\rightarrow\mathbb{R}_{+}$ such that%
\[
c(x,y)\leq a(x)+a(y).
\]
\end{condition}

Now consider $\mu,\nu\in L^{1}(a)\doteq\{\theta\in\mathcal{P}(S):\int
_{S}a(x)\theta(dx)<\infty\}$. Then
\begin{align*}
G(\mu\lVert\nu)  &  =\inf_{\gamma\in\mathcal{P}(S)}\left\{  W_{c}(\mu,\gamma)+R(\gamma\lVert
\nu)\right\} \\
&  \leq W_{c}(\mu,\nu)\\
&  =\inf_{\pi\in\Pi(\mu,\nu)}\left\{  \int_{S\times S}c(x,y)\pi(dx,dy)\right\}
\\
&  \leq\inf_{\pi\in\Pi(\mu,\nu)}\left\{  \int_{S\times S}\left[
a(x)+a(y)\right]  \pi(dx,dy)\right\} \\
&  =\int_{S}a(x)\mu(dx)+\int_{S}a(y)\nu(dy)\\
&  <\infty.
\end{align*}
We will assume the following mild conditions on the space $S$ and cost $c$ to make $\mathrm{Lip}(c,S;C_{b}(S))$ precompact.

\begin{condition}\label{Cond:c_cont}
There exists $\left\{K_m\right\}_{m\in\mathbb{N}}$ such that $K_m\subset S$ is compact, $K_{m}\subset K_{m+1}$ for all $m\in\mathbb{N}$, and $S = \cup_{m\in\mathbb{N}}K_m$. For each $m$, there exists $\theta_m: \mathbb{R}_+ \to \mathbb{R}_+$, such that $\lim_{a\to 0} \theta_m(a) = 0$, and $\delta_m >0$, such that for any $x,y\in K_m$ satisfying $d(x,y)\leq \delta_m$,
$$c(x,y)\leq \theta_m(d(x,y)).$$
\end{condition}
Recalling the definition (\ref{Lip_bdd}),
we define the unbounded version as follows
$$\mathrm{Lip}(c,S)\doteq\left\{  f\in C(S):f(x)-f(y)\leq
c(x,y)\mbox{ for all } x,y\in S\right\},$$
where $C(S)$ is the set of continuous functions mapping $S$ to $\mathbb{R}$. Before we proceed, we state the following lemma, which will be used repeatedly in this section.

\begin{lemma}\label{beyond_bounded}
If $g\in \mathrm{Lip}(c,S)$ and $\theta,\nu\in P(S)$ satisfy $\int_S |g|d\theta <\infty$,
then
$$\int_S gd\theta - \log \int_S e^g d\nu \leq G(\theta\lVert\nu)\leq R(\theta\lVert\nu).$$
\end{lemma}
\begin{proof}
We use a standard truncation argument. Since by Lemma \ref{basic} we already have $G(\theta\lVert\nu)\leq R(\theta\lVert\nu)$,
we only need to prove the first inequality in the statement of the lemma.
If $\int_S e^g d\nu = \infty$, then 
$$\int_S gd\theta - \log \int_S e^g d\nu = -\infty < 0\leq G(\theta\lVert\nu).$$
Hence we only need consider the case $\int_S e^g d\nu < \infty$. Let $g_n = \min(\max(g,-n),n)\in \mathrm{Lip}(c,S;C_b(S))=\Gamma$ for $n\in\mathbb{N}$. We have 
$|g_n(x)|\leq |g(x)|$
and 
$$\lim_{n\to\infty} g_n(x) = g(x) \quad x\in S.$$
Thus by the dominated convergence theorem
$$\lim_{n\to\infty}\int_S g_n d\theta = \int_S g d\theta.$$
Also we have 
$$e^{g_n(x)} \leq e^{g(x)}+1 \mbox{ and }
\lim_{n\to\infty} e^{g_n(x)} = e^{g(x)}.$$
Then again using the dominated convergence theorem, 
$$\lim_{n\to\infty}\int_S e^{g_n} d\nu = \int_S e^g d\nu.$$
Together with (\ref{eqn:var_forms}), this gives
\begin{align*}
\int_S gd\theta - \log\int_S e^gd\nu &= \lim_{n\to\infty} \left(\int_S g_nd\theta - \log\int_S e^{g_n} d\nu\right)\\
&\leq \sup_{f\in \Gamma}\left\{\int_S fd\theta - \log\int_S e^f d\nu\right\}\\
&=G(\theta\lVert\nu).
\end{align*}
\end{proof}

\vspace{\baselineskip}
Now we are ready to state the first main theorem of this section.

\begin{theorem}\label{optimizer} Suppose Conditions \ref{con:crep}, \ref{mea-det}, \ref{finite} and \ref{Cond:c_cont} are 
satisfied. Fix $\mu,\nu\in L^{1}(a)$.
Then the following conclusions hold.

\noindent1) There exists a unique optimizer $\gamma^{\ast}$ in the expression
(\ref{Vari}).

\noindent2) There exists an optimizer $g^{\ast}\in\mathrm{Lip}(c,S)$ in the
expression (\ref{Vari}), which is unique up to an additive constant in
$\mathrm{supp}(\mu)\cup\mathrm{supp}(\nu)$.

\noindent3) $g^{\ast}$ and $\gamma^{\ast}$ satisfy the following
conditions:

i)
\[
\frac{d\gamma^{\ast}}{d\nu}(x)=\frac{e^{g^{\ast}(x)}}{\int_{S} e^{g^{\ast}(y)}%
d\nu},\quad \nu-a.s.
\]

ii)
\[
W_{c}(\mu,\gamma^{*})=\int_{S} g^{*} d(\mu-\gamma^{*}).
\]

\end{theorem}
\begin{remark}
With many analogous expressions related to relative entropy, one can only conclude the uniqueness of $\gamma^*$ and $g^*$ (up to constant addition) almost everywhere according to either the measure $\mu$ or $\nu$. Moreover, because of the regularity condition $g^*\in\mathrm{Lip}(c,S;C(S))$ and Condition \ref{Cond:c_cont}, the uniqueness of $g^*$ (up to constant addition) on $\mathrm{supp}(\mu)\cup\mathrm{supp}(\nu)$ will follow.
\end{remark}

\begin{proof}
For $n\in\mathbb{N}$ consider $\gamma_{n}\in\mathcal{P}(S)$ that satisfies%
\[
R(\gamma_{n}\lVert\nu)+W_{c}(\mu,\gamma_{n})\leq G(\mu\lVert\nu)+\frac{1}{n}.
\]
Then by \cite[Lemma 1.4.3(c)]{dupell4} $\{\gamma_{n}\}_{n\geq1}$ is precompact in the weak topology, and thus
has a convergent subsequence $\left\{\gamma_{n_{k}}\right\}_{k\geq 1}$. Denote $\gamma^{\ast}\doteq
\lim_{k\rightarrow\infty}\gamma_{n_{k}}$. Then by the lower semicontinuity of
both $R(\cdot\lVert\nu)$ and $W_{c}(\mu,\cdot)$, we have%
\[
R(\gamma^{\ast}\lVert\nu)+W_c(\mu,\gamma^{\ast})\leq\liminf_{k\rightarrow\infty
}\left(  R(\gamma_{n_{k}}\lVert\nu)+W_{c}(\mu,\gamma_{n_{k}})\right)  \leq
G(\mu\lVert\nu).
\]
Since
\[
G(\mu\lVert\nu)=\inf_{\gamma\in\mathcal{P}(S)}\left\{  R(\gamma\lVert
\nu)+W_{c}(\mu,\gamma)\right\}  \leq R(\gamma^{\ast}\lVert\nu)+W_c(\mu
,\gamma^{\ast})
\]
it follows that%
\[
G(\mu\lVert\nu)=R(\gamma^{\ast}\lVert\nu)+W_c(\mu,\gamma^{\ast}),
\]
which shows that $\gamma^{\ast}$ is an optimizer in expression (\ref{Vari}).
If there exist two optimizers $\gamma_{1}\neq\gamma_{2}$, the strict convexity
of $R(\cdot\lVert\nu)$ and convexity of $W_{c}(\mu,\cdot)$ imply that for
$\gamma_{3}=\frac{1}{2}(\gamma_{1}+\gamma_{2})$%
\begin{align*}
R(\gamma_{3}\lVert\nu)+W_{c}(\mu,\gamma_{3})  &  <\frac{1}{2}\left(  \left(
R(\gamma_{1}\lVert\nu)+W_{c}(\mu,\gamma_{1})\right)  +\left(  R(\gamma
_{2}\lVert\nu)+W_{c}(\mu,\gamma_{2})\right)  \right) \\
&  =G(\mu\lVert\nu)\leq R(\gamma_{3}\lVert\nu)+W_{c}(\mu,\gamma_{3}),
\end{align*}
a contradiction. Thus the existence and uniqueness of an optimizer
$\gamma^{\ast}$ of (\ref{Vari}) is proved, which establishes 1) in the statement of the theorem. Before proceeding, we establish the following lemma.

\begin{lemma}\label{opt_integrability}
If $g\in \mathrm{Lip}(c,S)$, then
$$\int_S g d\gamma^* <\infty.$$
\end{lemma}
\begin{proof}
This can be shown by contradiction. Assume there exists $h \in \mathrm{Lip}(c,S)$ such that $\int_S |h| d\gamma^*=\infty$. By symmetry, we can  just consider $h$ to be non-negative, since $\max (h,0)\in\mathrm{Lip}(c,S)$ and $h=\max(h,0) - \max(-h,0)$. Thus we can assume there exists non-negative $h\in \mathrm{Lip}(c,S)$ satisfying
$$\int_S h d\gamma^* = \infty,$$
and by the fact that $u\in L^1(a)$ together with Condition \ref{finite}, 
\begin{align*}
\int_S h d\mu &\leq \int_S \left[ h(0) + c(x,0)\right] \mu(dx)\\
&= h(0) + a(0) +\int_S a(x) \mu(dx)<\infty.
\end{align*}
Then
\begin{align*}
W_c(\mu,\gamma^*) & = \sup_{g \in \mathrm{Lip(c,S)}}\int_S g d(\mu - \gamma^*)\\
&\geq \limsup_{n\to\infty}\int_S \max(-h,-n) d(\mu-\gamma^*)\\
& = \limsup_{n\to\infty}\left[ \int_S \max(-h,-n) d\mu + \int_S \min(h,n)d\gamma^*\right] \\
& = \int_S -h d\mu + \int_S h d\gamma^*\\
&= \infty,
\end{align*}
where the second to last equation comes from dominated and monotone convergence theorems applied to the first and second terms respectively. However, since $\gamma^*$ is the optimizer, we have 
$$W_c(\mu,\gamma^*) \leq W_c(\mu,\gamma^*) + R(\gamma^*\lVert \nu) = G(\mu\lVert \nu) <\infty.$$
This contradiction shows the integrability of $\gamma^*$ with respect to any $\mathrm{Lip}(c,S)$ function. 
\end{proof}

\vspace{\baselineskip}
Now we consider the other variational representation of $G(\mu\lVert \nu)$, which is 
$$G(\mu\lVert \nu) = \sup_{g\in\mathrm{Lip}(c,S;C_b(S))}\left\{\int_S gd\mu-\log\int_S e^{g}d\nu\right\}.$$
Take $g_n\in\mathrm{Lip}(c,S;C_b(S))$ such that 
$$ G(\mu\lVert\nu) - 1/n \leq \int_S g_n d\mu-\log\int_S e^{g_n}d\nu \leq G(\mu\lVert \nu) .$$
Without loss of generality, we can  assume $g_n(x_0)=0$ for some fixed $x_0\in K_0\subset S$.  Since for any $m\in\mathbb{N}$ $K_m\subset S$  is compact, we have that $\left\{g_n\right\}_{n\in\mathbb{N}}$ is bounded and equicontinuous on $K_m$ by Condition \ref{Cond:c_cont}. By the  Arzel\`{a}-Ascoli theorem, there exists a subsequence of $\left\{g_n\right\}_{n\in\mathbb{N}}$ that converges uniformly in $K_m$. Using diagonal argument, by taking subsequences sequentially along $\left\{K_m\right\}_{m\in\mathbb{N}}$, where the next subsequence is a subsequence of the former one, and take one element from each sequence, we conclude there exists a subsequence $\left\{g_{n_j}\right\}_{j\in\mathbb{N}}$, that converges uniformly in any $K_m$. Since $S=\cup_{m\in\mathbb{N}}K_m$, we conclude that $\left\{g_{n_j}\right\}_{j\in\mathbb{N}}$ converges pointwise in $S$. Denotes its limit by $g^*$. It can be easily verified that $g^*\in \mathrm{Lip}(c,S)$.
 
Since $g_{n_j}(x)\leq g_{n_j}(x_0) + c(x_0,x)\leq a(x_0)+a(x)$ and $\int_S\left( a(x_0)+a(x)\right)d\mu <\infty$, by the dominated convergence theorem 
$$\lim_{j\to\infty} \int_S g_{n_j} d\mu = \int_S g^* d\mu.$$
By Fatou's lemma, we have 
$$\liminf_{j\to\infty} \int_S e^{g_{n_j}}d\nu \geq \int e^{g^*}d\nu,$$
and therefore 
$$-\log\int e^{g^*}d\nu\geq\limsup_{j\to\infty} -\int_S e^{g_{n_j}}d\nu.$$

Putting these together, we have 
\begin{align*}
G(\mu\lVert \nu) &= \sup_{g\in\mathrm{Lip}(c,S;C_b(S))}\left\{\int_S gd\mu-\log\int_S e^{g}d\nu\right\}\\
&\leq \limsup_{j\to\infty} \left\{\int_S g_{n_j}d\mu-\log\int_S e^{g_{n_j}}d\nu\right\}\\
&\leq \int_S g^* d\mu - \log\int_S e^{g^*} d\nu\\
&= \left(\int_S g^* d\mu - \int_S g^* d\gamma^* \right)+ \left(\int_S g^* d\gamma^* - \log\int_S e^{g^*} d\nu\right).
\end{align*}
We can add and subtract $\int_S g^*d\gamma^*$ because we have proved in Lemma \ref{opt_integrability} that $\gamma^*$ is integrable with respect to functions in $\mathrm{Lip}(c,S)$, and $g^*\in \mathrm{Lip}(c,S)$. 
By Lemma \ref{beyond_bounded} we have
$$\int_S g^* d\gamma^* - \log\int_S e^{g^*} d\nu \leq R(\gamma^*\lVert\nu). $$
We also have 
$$\int_S g^* d\mu - \int_S g^* d\gamma^* \leq  W_c(\mu,\gamma^*),$$
which is due to 
\begin{align*}
W_c(\mu,\gamma^*)&= \sup_{g\in\mathrm{Lip}(c,S;C_b(S))}\int_S gd(\mu-\gamma^*)\\
&\geq \limsup_{n\to\infty}\int_S\max(\min(g^*,n),-n)d(\mu-\gamma^*)\\
&= \int_S g^*d(\mu-\gamma^*),
\end{align*}
where the last equality is because of the dominated convergence theorem and integrability of $|g^*|$ with respect to $\mu$ and $\gamma^*$ (Lemma \ref{opt_integrability}).
We can therefore continue the calculation above as 
\begin{align*}
 &\left(\int_S g^* d\mu - \int_S g^* d\gamma^* \right)+ \left(\int_S g^* d\gamma^* - \log\int_S e^{g^*} d\nu\right)\\
&\qquad\leq W_c(\mu,\gamma^*) + R(\gamma^*\lVert\nu)\\
&\qquad= G(\mu\lVert\nu).
\end{align*}

Since both the upper and lower bounds on the inequalities coincide, we must have all inequalities to be equalities, and therefore
$$G(\mu\lVert \nu) = \int_S g^*d\mu - \log\int_S e^{g^*} d\nu,$$
$$\int_S g^* d\mu - \int_S g^* d\gamma^* =  W_c(\mu,\gamma^*),$$
and 
$$\int_S g^* d\gamma^* - \log\int_S e^{g^*} d\nu = R(\gamma^*\lVert\nu). $$
The last equation gives us the relationship
$$\frac{d\gamma^*}{d\nu}(x) = \frac{e^{g^*(x)}}{\int_S e^{g^*}d\nu} \quad \nu-a.s.$$

 Thus we have shown the existence of optimizer $g^*\in\mathrm{Lip}(c,S)$ and its relationship with $\gamma^*$. Lastly, for any other optimizer $\bar{g}\in\mathrm{Lip}(c,S)$ the analogous argument shows
%
$$\frac{d\gamma^*}{d\nu}(x) = \frac{e^{\bar{g}(x)}}{\int_S e^{\bar{g}}d\nu} \quad \nu-a.s.$$
Hence  uniqueness of the optimizer $g^*$ in $\mathrm{supp}(\nu)$ up to $\nu-a.s.$ is also proved. 

To determine the uniqueness of the optimizer $g^*$ in $\mathrm{supp}(\mu)$, we take an optimal transport plan between $\mu$ and $\gamma^*$, $\pi^{*}\in
\Pi(\mu,\gamma^{*})$ for $W_{c}(\mu,\gamma^{*})$, which means
$$W_c(\mu,\gamma^*) = \inf_{\pi\in\Pi(\mu,\gamma^*)}\left\{\int_{S\times
S}c(x,y)\pi(dx,dy)\right\}=\int_{S\times
S}c(x,y)\pi^*(dx,dy).$$
(Note that $c$ satisfying Condition \ref{con:crep} is lower semicontinuous, and therefore   \cite[Theorem 1.5]{ambgig} shows the existence of an optimal transport plan $\pi^*$.)

Since $g^*(x)-g^*(y)\leq c(x,y)$, 
\begin{align*}
W_c(\mu,\gamma^*) &= \int_{S\times
S}c(x,y)\pi^*(dx,dy)\\
 &\geq \int_{S\times
 S}\left[ g^*(x)-g^*(y)\right] \pi^*(dx,dy)\\
&= \int_S g^*(x) (\mu-\gamma^*)(dx)\\
&= W_c(\mu,\gamma^*).
\end{align*}
Then the only inequality above must be equality, which implies that for $(x,y)\in\mathrm{supp}(\gamma^*)$, $g^*(x)-g^*(y)=c(x,y)$, $\pi^*- a.s.$ This is also true for any other optimizer $\bar{g}\in \mathrm{Lip}(c,S)$ for (\ref{Vari}). Thus we are able to determine $g^*$ uniquely in $\mathrm{supp}(\mu)$ $\mu-a.s.$ with the help of $\pi^*$ and data of $g^*$ in $\mathrm{supp}(\nu)$. Lastly, since $g^*\in\mathrm{Lip}(c,S)$ and by Condition \ref{Cond:c_cont}, we conclude the uniqueness of $g^*$ in $\mathrm{supp}(\mu)\cup\mathrm{supp}(\nu)$ by the continuity of $g^*$.
\end{proof}

\begin{remark}
When $\mu\ll\nu$ Theorem \ref{optimizer} implies that for some constant $c_0$%
\[
g^{\ast}(x)=\log\left(  \frac{d\gamma^{\ast}}{d\nu}(x)\right)  -c_{0}%
\quad \nu-a.s.
\]
Hence 
\[
G(\mu\lVert\nu)=\int_{S}g^{\ast}d\mu-\log\int_{S}e^{g^{\ast}}d\nu=\int_{S}
\log\left(  \frac{d\gamma^{\ast}}{d\nu}(x)\right)  d\mu,
\]
and so the $\Gamma$-divergence of $\mu$ with respect to $\nu$ looks like a \textquotedblleft modified\textquotedblright\ version of relative entropy.
\end{remark}

The next theorem tells us that 3) of Theorem \ref{optimizer} is not only a
description of of the pair of optimizer $(g^{*},\gamma^{*})$, but also a
characterization of it.

\begin{theorem}
\label{verif} 
Suppose Conditions \ref{con:crep}, \ref{mea-det}, \ref{finite} and \ref{Cond:c_cont} are 
satisfied. Fix $\mu,\nu\in L^{1}(a)$. If $g_{1}\in$\emph{ Lip}$(c,S)$ and
$\gamma_{1}\in\mathcal{P}(S)$ satisfy condition 3) in Theorem
\ref{optimizer}, then $(g_{1},\gamma_{1})$ are optimizers in the corresponding
variational problem (\ref{Vari}).
\end{theorem}

\begin{proof}
The theorem follows from the two variational characterization of $\Gamma$-divergence
in (\ref{Vari}).
Condition 3) of Theorem \ref{optimizer} implies
\[
R(\gamma_{1}\lVert\nu)=\int_{S}g_{1}d\gamma_{1}-\log\int_{S}e^{g_{1}}d\nu
\mbox{ and }
W_{c}(\mu,\gamma_{1})=\int_{S}g_{1}d(\mu-\gamma_{1}),
\]
and therefore 
\[
R(\gamma_{1}\lVert\nu)+W_{c}(\mu,\gamma_{1})=\int_{S}g_{1}d\mu-\log\int
_{S}e^{g_{1}}d\nu.
\]
This implies
\begin{align*}
G(\mu\lVert\nu)  &  =\inf_{\gamma\in\mathcal{P}(S)}\left\{  R(\gamma\lVert
\nu)+W_{c}(\mu,\gamma)\right\} \\
&  \leq R(\gamma_{1}\lVert\nu)+W_{c}(\mu,\gamma_{1})\\
&  =\int_{S}g_{1}d\mu-\log\int_{S}e^{g_{1}}d\nu\\
&  \leq\sup_{g\in\Gamma}\left\{  \int_{S}gd\mu-\log\int_{S}e^{g}d\nu\right\}
\\
&  =G(\mu\lVert\nu).
\end{align*}
The first inequality comes from the fact that $\gamma_1\in \mathcal{P}(S)$, while the second needs a little more discussion, which will be given below. Assuming this, the last display shows that 
$(g_{1},\gamma_{1})$ are optimizers. The second inequality follows from Lemma \ref{beyond_bounded} and the fact that 

\begin{align*}
\int_S |g_1(x)|\mu(dx) &\leq \int_S |g_1(0)|+c(0,x)\mu(dx)\\
&\leq \int_S |g_1(0)|+a(0)+a(x)\mu(dx)<\infty.
\end{align*}
The proof is complete.
\end{proof}

\vspace{\baselineskip}
The last theorem answers questions i) and ii) raised earlier in this section, now we want to answer iii), which is to characterize the directional derivative of $G(\mu\lVert\nu)$ in the first variable when fixing the second one, i.e., 
$$\lim_{\varepsilon\to 0^+} \frac{1}{\varepsilon}\left(G(\mu+\varepsilon \rho\lVert \nu)-G(\mu\lVert\nu)\right)$$
for $\rho\in\mathcal{M}_0(S)$ which satisfies certain conditions. From Theorem \ref{optimizer} and remarks following it  we know that any optimizer $g^*$ of expression (\ref{Vari}) is unique in $\mathrm{supp}(\mu)\cup \mathrm{supp}(\nu)$. However, there is still freedom to choose $g^*$ in $S\backslash\left\{\mathrm{supp}(\mu)\cup \mathrm{supp}(\nu)\right\}$, since the variational problem in (\ref{Vari}) does not take into account of the information of $g^*$ outside $\mathrm{supp}(\mu)\cup \mathrm{supp}(\nu)$, other than requiring that $g^*$ belong to $\mathrm{Lip}(c,S)$. We will define a special $g^*$ that is uniquely defined not only in $\mathrm{supp}(\mu)$ and $\mathrm{supp}(\nu)$, but also on $S\backslash\left\{\mathrm{supp}(\mu)\cup \mathrm{supp}(\nu)\right\}$. For $x\in S\backslash\left\{\mathrm{supp}(\mu)\cup \mathrm{supp}(\nu)\right\}$, set
\begin{align}\label{the_opt}
g^*(x)\doteq\inf_{y\in\mathrm{supp}(\nu)}\left\{g^*(y)+c(x,y)\right\}.
\end{align}
From now on we will use the notation $g^*$ for the function defined in \eqref{the_opt}. The following lemma 
confirms that this construction of $g^*$ still lies in $\mathrm{Lip}(c,S)$.

\begin{lemma}
The following two statements hold.

1) For $x\in \mathrm{supp}(\mu)$, the expression $(\ref{the_opt})$ also holds. In other words, for $x\in S\backslash\mathrm{supp}(\nu)$, we have 

$$g^*(x) = \inf_{y\in\mathrm{supp}(\nu)}\left\{g^*(y)+c(x,y)\right\}.$$

2) $g^*$ defined by equation (\ref{the_opt}) is in $\mathrm{Lip}(c,S)$. In addition, 
\begin{align}\label{biggest}
g^*(x) = \sup\{h(x): h\in \mathrm{Lip}(c,S), h(y) = g^*(y)\ \mathrm{for}\ y\in\mathrm{supp}(\nu) \}
\end{align}
\end{lemma}
\begin{proof}

1) For $x\in\mathrm{supp}(\mu)$, from an optimal transport plan between $\mu$ and $\gamma^*$, $\pi^*\in\Pi(\mu,\gamma^*)$ for $W_c(\mu,\gamma^*)$, we know there exists $y_x\in \mathrm{supp}(\nu)$ such that $(x,y_x)\in\mathrm{supp}(\pi^*)$. Thus by \cite{ambgig}[Remark 1.15],
$$g^*(x)= g^*(y_x) +c(x,y_x).$$
On the other hand, by Theorem \ref{optimizer}, $g^*|_{\mathrm{supp}(\nu)\cup\mathrm{supp}(\mu)}\in\mathrm{Lip}(c,S)$. Thus, for other $y\in\mathrm{supp}(\nu)$, $g^*(x)\leq c(x,y)+g^*(y)$, which in turn gives

$$g^*(x) \leq \inf_{y\in\mathrm{supp}(\nu)}\left\{g^*(y)+c(x,y)\right\}.$$
By combining the two expressions above, we have for $x\in\mathrm{supp}(\mu)$, (\ref{the_opt}) also holds. In other words, $g^*$ is totally characterized by $g^*|_\mathrm{supp}(\nu)$ and (\ref{the_opt}).

2) Since $c\geq 0$, it is easily checked that for any $x\not\in \mathrm{supp}(\nu)$ and any $y\in\mathrm{supp}(\nu)$,
$$g^*(y)\leq g^*(x) \leq g^*(y)+c(x,y).$$
For $x\in\mathrm{supp}(\nu)$, since we already know $g^*|_{\mathrm{supp}(\nu)}$ is uniquely determined and the optimizer constructed in Theorem \ref{optimizer} is in $\mathrm{Lip}(c,S)$, we  conclude that for any $y\in\mathrm{supp}(\nu)$, 
$$g^*(y)-c(x,y)\leq g^*(x) \leq g^*(y)+c(x,y).$$
Hence to show $g^*\in\mathrm{Lip}(c,S)$ we only need to check for $x_1,x_2\not\in \mathrm{supp}(\nu)$ the Lipschitz constrait is satisfied. 
From  the definition (\ref{the_opt}),
we know for any $n<\infty$
there exists $y_1\in \mathrm{supp}(\nu)$ such that 
$$c(x_1,y_1) - 1/n \leq g^*(x_1)-g^*(y_1).$$
Also, because $y_1\in\mathrm{supp}(\nu)$,
$$g^*(x_2)-g^*(y_1) \leq c(x_2,y_1).$$
Therefore 
\begin{align*}
g^*(x_2)-g^*(x_1)&\leq (c(x_2,y_1)-c(x_1,y_1))+1/n\\
&\leq c(x_1,x_2)+1/n,
\end{align*}
where the last inequality uses the triangle inequality property of $c$. Since $n>0$ is arbitrary and we can swap the roles of $x_1$ and $x_2$, we have proved the Lipschitz condition of $g^*$ for $x_1,x_2\not\in\mathrm{supp}(\nu)$. Thus the statement that $g^*\in \mathrm{Lip}(c,S)$ is proven.

For (\ref{biggest}), notice that for $h\in\mathrm{Lip}(c,S)$,  $x\in S$ and $y\in\mathrm{supp}(\nu)$,
$$h(x) \leq h(y)+ c(x,y).$$
So if $h(y) = g^*(y)$ for $y\in\mathrm{supp}(\nu)$, then for $x\in S\backslash\mathrm{supp}(\nu)$,
$$h(x) \leq \inf_{y\in\mathrm{supp}(\nu)}\left\{h(y)+c(x,y)\right\} = \inf_{y\in\mathrm{supp}(\nu)}\left\{g^*(y)+c(x,y)\right\}=g^*(x).$$
Since $g^*$ is also in $\mathrm{Lip}(c,S)$, this proves (\ref{biggest}).
\end{proof}

\vspace{\baselineskip}
Then based on this construction, we have the following result.  

\begin{theorem}\label{first_variation}
Take $\Gamma=\mathrm{Lip}(c,S;C_b(S))$ where $c$ satisfies the conditions of Theorem \ref{optimizer} and $\mu,\nu\in L^1(a)$. Take $\rho=\rho_+-\rho_-\in\mathcal{M}_0(S)$ where $\rho_+,\rho_-\in\mathcal{P}(S)$ are mutually singular probability measures, $\rho_+\in L^1(a)$, and assume there exists $\varepsilon_0>0$ such that $\mu+\varepsilon\rho\in\mathcal{P}(S)$ for $0<\varepsilon\leq\varepsilon_0$. Then
$$\lim_{\varepsilon\to 0^+} \frac{1}{\varepsilon}\left(G(\mu+\varepsilon \rho\lVert \nu)-G(\mu\lVert\nu)\right)=\int_S g^* d\rho.$$
where $g^*$ is the optimizer found in (\ref{the_opt}). 
\end{theorem}

\begin{proof}
We use the variational formula (\ref{Vari}) for $G(\mu+\varepsilon\rho\lVert \nu)$, where $\mu+\varepsilon\rho\in\mathcal{P}(S)$ and $\rho_+\in L^1(a)$. Recall that $g^*$ is the optimizer for (\ref{Vari}). Using Lemma \ref{beyond_bounded} with $\theta = \mu+\varepsilon\rho$,
\begin{align*}
G(\mu+\varepsilon\rho\lVert \nu) &=\sup_{g\in\Gamma}\left\{\int_{S}g d(\mu+\varepsilon\rho)-\log\int_{S}e^{g}%
d\nu\right\} \\
&\geq \int_S g^*d(\mu+\varepsilon\rho) - \log\int_S e^{g^*} d\nu\\
&= \varepsilon\int_S g^*d\rho + \int_S g^*d\mu-\log\int_S e^{g^*}d\nu\\
&=\varepsilon \int_S g^*d\rho + G(\mu\lVert\nu). 
\end{align*}
Thus
\begin{equation}
    \liminf_{\varepsilon\to 0^+} \frac{1}{\varepsilon}\left(G(\mu+\varepsilon \rho\lVert \nu)-G(\mu\lVert\nu)\right)\geq\int_S g^* d\rho.
    \label{eqn:LB}
\end{equation}

The other direction is more delicate. Take $f(\varepsilon) = G(\mu+\varepsilon\rho\lVert\nu)$. From Lemma \ref{basic} we know that $f$ is convex, lower semicontinuous and finite on $[0,\varepsilon_0]$. Using a property of convex functions in one dimension, we know $f$ is differentiable on $(0,\varepsilon_0)$ except for a countable number of points. Take $\varepsilon\in(0,\varepsilon_0)$ to be a place where $f$ is differentiable, and $\delta>0$ small. Take $g^*_\varepsilon\in \mbox{Lip}(c,S)$ to be the optimizer for $G(\mu+\varepsilon\rho\lVert\nu)$ satisfying $g^*_\varepsilon(0)=0$, so that
$$G(\mu+\varepsilon\rho\lVert\nu)=\int_S g^*_\varepsilon d(\mu+\varepsilon\rho) - \log\int_S e^{g^*_\varepsilon} d\nu.$$
Then using an argument that already appeared in this proof, we have 
$$G(\mu+(\varepsilon+\delta)\rho\lVert\nu) - G(\mu+\varepsilon\rho\lVert\nu )\geq \delta\int_S g^*_\varepsilon d\rho,$$
and 
$$G(\mu+(\varepsilon-\delta)\rho\lVert\nu) - G(\mu+\varepsilon\rho\lVert\nu)\geq -\delta\int_S g^*_\varepsilon d\rho.$$
It follows that 
\begin{align*}
\int_S g^*_\varepsilon d\rho&\leq \lim_{\delta\to 0} \frac{1}{\delta} \left(G(\mu+(\varepsilon+\delta)\rho\lVert\nu) - G(\mu+\varepsilon\rho\lVert\nu ) \right)\\
&= f'(\varepsilon)\\
&= \lim_{\delta\to 0}\frac{1}{\delta} \left(G(\mu+\varepsilon\rho\lVert\nu) - G(\mu+(\varepsilon - \delta)\rho\lVert\nu ) \right) \\
&\leq \int_S g^*_\varepsilon d\rho.
\end{align*}
and therefore 
\begin{equation}
    f'(\varepsilon) = \int_S g^*_\varepsilon d\rho.\label{eqn:fderiv}
\end{equation}

If we denote 
$$f'_+(0) = \lim_{\varepsilon\to 0^+} \frac{1}{\varepsilon}(f(\varepsilon)-f(0)),$$
then by a property of convex functions \cite[Theorem 24.1]{roc},
for any sequence of $\left\{\varepsilon_n\right\}_{n\in\mathbb{N}}$ such that $\varepsilon_0 >  \varepsilon_n\downarrow 0$ and $f$ is differentiable at $\varepsilon_n>0$, we have 
$$f'_+(0)  = \lim_{n\to\infty}f'(\varepsilon_n) = \lim_{n\to\infty} \int_S g^*_{\varepsilon_n}d\rho.$$
By the same argument used in the proof of Theorem \ref{optimizer} (paragraphs following Lemma \ref{opt_integrability}), i.e., by applying the Arzel\`{a}-Ascoli theorem to $\{g_{\varepsilon_n}\}$ on each compact set $K_m\subset S$, and then doing a diagonalization argument, 
there exists a subsequence of $\left\{n_k\right\}_{k\geq 0} \subset \left\{n\right\}_{n\geq 0}$, such that $g_{\varepsilon_{n_k}}^*$ converges pointwise to a function that we denote by $g_0^*\in\mathrm{Lip}(c,S)$. To simplify the notation, let $n$ denote the convergent subsequence.

Since $\rho = \rho_+ - \rho_-$, where $\rho_+\in L^1(a)$ and $\mu+\varepsilon_0 \rho\in P(S)$, $\mu\in L^1(a)$ implies $\rho_-\in L^1(a)$, therefore
$$\int_S a d|\rho| < \infty.$$
Here $|\rho|=\rho_+ + \rho_-$. Recall that for any $\varepsilon\in(0,\varepsilon_0)$, $g^*_\varepsilon(0)=0$. For any $x\in S$,
$$g^*_\varepsilon(x)\leq g^*_{\varepsilon}(0) + c(0,x) \leq a(0) + a(x).$$
Thus by the dominated convergence theorem
$$f'_+(0) = \lim_{n\to\infty} \int _S g^*_{\varepsilon_{n}} d\rho = \int_S g^*_0 d\rho.$$

Lastly, to connect $g_0^*$ back to $g^*$ defined in (\ref{the_opt}), note that by the lower semincontinuity of $G(\cdot\lVert\nu)$,
\begin{align*}
G(\mu\lVert\nu)&\leq \liminf_{n\to\infty} G(\mu+\varepsilon_{n}\rho\lVert\nu)\\
& = \liminf_{n\to\infty}\left( \int_S g^*_{\varepsilon_{n}}d(\mu+\varepsilon_{n}\rho) - \log\int_S e^{g^*_{\varepsilon_{n}}}d\nu\right)\\
& = \liminf_{n\to\infty} \int_S g^*_{\varepsilon_{n}}d(\mu+\varepsilon_{n}\rho) - \limsup_{n\to\infty}\log\int_S e^{g^*_{\varepsilon_{n}}}d\nu\\
&\leq \int_S g^*_0 d\mu -\log\int_S e^{g_0^*}d\nu \\
&\leq G(\mu\lVert\nu).
\end{align*}
The second inequality uses dominated convergence, \eqref{eqn:fderiv},
and that by Fatou's lemma
$$\limsup_{n\to\infty} \int_S e^{g^*_{\varepsilon_{n}}}d\nu\geq\liminf_{n\to\infty} \int_S e^{g^*_{\varepsilon_{n}}}d\nu \geq \int_S e^{g^*_0}d\nu.$$
The third inequality uses Lemma \ref{beyond_bounded}.

Since both sides of the inequality coincide, $g^*_0$ must be the optimizer for variational expression (\ref{Vari}). By Theorem \ref{optimizer} and  equation (\ref{biggest}),
we have $g^*_0(x) \leq g^*(x)$ for all $x\in S$.
Thus
\begin{equation}
 f'_+(0)=\int_Sg^*_0d\rho \leq \int_Sg^*d\rho,  \label{eqn:UB}
\end{equation}
the other direction of the inequality is proved. 
Combining \eqref{eqn:UB} and \eqref{eqn:LB} gives 
$$\lim_{\varepsilon\to 0^+} \frac{1}{\varepsilon}\left(G(\mu+\varepsilon\rho\lVert\nu) - G(\mu\lVert\nu)\right)=\int_S g^*d\rho.$$
\end{proof}

\begin{remark}
When $\rho\in\mathcal{M}_0(S)$ is taken such that there exists $\varepsilon_0>0$ such that for $\varepsilon\in[-\varepsilon_0,\varepsilon_0]$, $\mu+\varepsilon\rho\in P(S)$, then by applying the above theorem to $\rho$ and $-\rho$ respectively, we can conclude $G(\mu+\varepsilon\rho\lVert\nu)$ as a function of $\varepsilon$ is differentiable at $\varepsilon = 0$ with derivative $\int_S g^*d\rho$.
\end{remark}

\begin{remark}
We call $g^*$ defined in (\ref{the_opt}) the unique potential associated with $G(\mu\lVert\nu)$. This $g^*$ is similar to the Kantorovich potential in the optimal transport literature. However, for the optimal transport cost $W_c(\mu,\nu)$ more conditions are needed(e.g. \cite{san1}[Proposition 7.18]) to ensure the uniqueness of the Kantorovich potential. Here under very mild conditions we are able to confirm the uniqueness of the potential, and prove that it is the directional derivative of the corresponding $\Gamma$-divergence, as is case of the Kantorovich potential for optimal transport cost when its uniqueness is established.  
\end{remark}

\section{Limits and Approximations of $\Gamma$-divergence}

In this section, we consider limits that are obtained as the admissible set gets large or small,
and the $\Gamma$-divergence will be approximated by relative entropy or a transport distance, respectively. We also consider in special cases more informative expansions.
Throughout the section we assume the conditions of Theorem \ref{optimizer}.

Fix an admissible set of $\Gamma_0 \subset C_b(S)$, and take $\Gamma = b\Gamma_0 =\left\{b\cdot  f:f\in\Gamma_0\right\}$ for $b>0$. Then the following proposition holds.
\begin{proposition}\label{lim_RE}
For $\mu,\nu\in \mathcal{P}(S)$, 
$$\lim_{b\to\infty} G_{b\Gamma_0}(\mu\lVert\nu) = R(\mu\lVert\nu).$$
\end{proposition}
\begin{proof}
We separate the proof into two cases, $R(\mu\lVert\nu)<\infty$ and $R(\mu\lVert\nu)=\infty$.

\vspace{\baselineskip}\noindent
1) If $R(\mu\lVert\nu)<\infty$, then for any $b>0$, 
\begin{align}
G_{b\Gamma_0}(\mu\lVert\nu) = \inf_{\gamma\in\mathcal{P}(S)}\left\{W_{b\Gamma_0}(\mu,\gamma)+R(\gamma\lVert\nu)\right\}\leq R(\mu\lVert\nu)<\infty.\label{eqn:bbounds}
\end{align}
From Theorem \ref{optimizer} we know there exists a unique optimizer $\gamma^*$ for each $b$, which we write as $\gamma_b^*$. Note that
 $$R(\gamma_b^*\lVert \nu)
 \leq R(\mu\lVert\nu)<\infty,$$
and therefore  $\left\{\gamma_b^*\right\}_{b>0}$ is precompact in the weak topology \cite[Lemma 1.4.3(c)]{dupell4}.
 Given any subsequence $b_k$, there exists a further subsequence (again denoted by $b_k$) and $\gamma_{\infty}^*\in\mathcal{P}(S)$ such that $\gamma^*_{b_k}\Rightarrow\gamma^*_\infty$. On the other hand,
 \begin{align*}
 W_{b\Gamma_0}(\mu,\gamma^*_b) &= \sup_{f\in b\Gamma_0}\left\{\int_S f d(\mu-\gamma_b^*)\right\}\\
 &=b\sup_{f\in \Gamma_0}\left\{\int_S f d(\mu-\gamma_b^*)\right\}=b W_{\Gamma_0}(\mu,\gamma_b^*),
 \end{align*}
 and 
 $W_{b\Gamma_0}(\mu,\gamma_b^*)\leq G_{b\Gamma_0}(\mu\lVert\nu) \leq R(\mu\lVert\nu)<\infty.$
 Thus
 \begin{align*}
 W_{\Gamma_0}(\mu,\gamma_\infty^*)&\leq \liminf_{k\to\infty}W_{\Gamma_0}(\mu,\gamma_{b_k}^*) = \liminf_{k\to\infty}\frac{1}{b_k}W_{b_k\Gamma_0}(\mu,\gamma_{b_k}^*)\\
 &\leq \liminf_{k\to\infty}\frac{1}{b_k}R(\mu\lVert\nu)=0,
 \end{align*}
and since $\Gamma_0$ is admissible,
 $\gamma_\infty^*=\mu$.  We thus conclude that
 \begin{align*}
 \liminf_{k\to\infty}G_{b_k\Gamma_0}(\mu\lVert\nu)&=\liminf_{k\to\infty}\left(W_{b_k\Gamma_0}(\mu,\gamma_{b_k}^*)+R(\gamma_{b_k}^*\lVert \nu)   \right)\\
 &\geq \liminf_{k\to\infty}R(\gamma_{b_k}^*\lVert \nu)\\
 & \geq R(\mu\lVert\nu),
 \end{align*}
 and since the original subsequence was arbitrary  
 $$\liminf_{b\to\infty}G_{b\Gamma_0}(\mu\lVert\nu)\geq R(\mu\lVert\nu).$$
 On the other hand, we have by \eqref{eqn:bbounds} that
 $$\limsup_{b\to\infty}G_{b\Gamma_0}(\mu\lVert\nu)\leq R(\mu\lVert\nu),$$
 and the statement is proved.
 
\vspace{\baselineskip}\noindent
 2) $R(\mu\lVert\nu)=\infty.$ For this case, we want to prove that 
 $$\liminf_{b\to\infty}G_{b\Gamma_0}(\mu\lVert\nu)=\infty.$$
 If not, then there exists a subsequence $\left\{b_k\right\}_{b\in\mathbb{N}}$ such that 
 $$\lim_{k\to\infty}G_{b_k\Gamma_0}(\mu\lVert\nu)<\infty.$$
 For this subsequence, we can apply the argument used in part 1) to conclude there exists $\gamma_{b_k}^*$ such that 
 $$G_{b_k\Gamma_0}(\mu\lVert\nu)=W_{b_k\Gamma_0}(\mu,\gamma_{b_k}^*)+R(\gamma_{b_k}^*\lVert\nu)  .$$
Moreover there exists a further subsequence of this sequence, which for simplicity we also denote by $\left\{b_k\right\}_{k\in\mathbb{N}}$, which satisfies $\gamma_{b_k}^*\Rightarrow \mu$. Then by the same argument as in 1), we would conclude
$$\lim_{k\to\infty}G_{b_k\Gamma_0}(\mu\lVert\nu)\geq R(\mu\lVert\nu)=\infty.$$
This contradiction proves the statement.
\end{proof}

\vspace{\baselineskip}
On the other hand, if $\Gamma=\delta\Gamma_0$ for small $\delta>0$, we can approximate the $\Gamma$-divergence in terms of the $W_{\Gamma_0}$.

\begin{proposition}
For $\mu,\nu\in\mathcal{P}(S)$
$$\lim_{\delta\to 0}\frac{1}{\delta} G_{\delta\Gamma_0}(\mu\lVert\nu)=W_{\Gamma_0}(\mu,\nu).$$
\end{proposition}
\begin{proof}
For any $\delta>0$, Jensen's inequality implies
\begin{align*}
\frac{1}{\delta}G_{\delta \Gamma_0}(\mu\lVert\nu) &= \frac{1}{\delta} \sup_{g\in\delta\Gamma_0}\left\{\int_S g d\mu - \log\int_S e^gd\nu \right\}\\
&\leq \frac{1}{\delta} \sup_{g\in\delta\Gamma_0}\left\{\int_S g d\mu - \int_S g d\nu\right\}\\
&=\sup_{g\in\Gamma_0}\left\{\int_S g d\mu - \int_S g d\nu\right\}\\
&= W_{\Gamma_0}(\mu,\nu),
\end{align*}
and therefore 
$$\limsup_{\delta\to 0}\frac{1}{\delta}G_{\delta \Gamma_0}(\mu\lVert\nu)\leq W_{\Gamma_0}(\mu,\nu).$$

For the reverse inequality we consider two cases.

\vspace{\baselineskip}\noindent
1) $W_{\Gamma_0}(\mu,\nu)<\infty.$ For $0<\delta<1$ the argument used above shows  
$$G_{\delta\Gamma_0}(\mu\lVert\nu)\leq \delta W_{\Gamma_0}(\mu,\nu) \leq W_{\Gamma_0}(\mu,\nu)
<\infty.$$
By Theorem \ref{optimizer}, we know there exists $\gamma^*_\delta\in\mathcal{P}(S)$, such that 
$$G_{\delta\Gamma_0}(\mu\lVert\nu) = W_{\delta\Gamma_0}(\mu,\gamma^*_\delta)+R(\gamma^*_\delta\lVert\nu).$$
Since $R(\gamma^*_\delta\lVert\nu)<G_{\delta\Gamma_0}(\mu\lVert\nu)\leq W_{\Gamma_0}(\mu,\nu)$ for $\delta\in(0,1)$, for any sequence $\delta_k\subset (0,1)$ there a further a subsequence (again denoted $\delta_k$) such that $\delta_k$ is decreasing, $\lim_{k\to\infty}\delta_k=0$, and $\gamma^*_{\delta_k}$ converges weakly to a probability measure, which we denote as $\gamma_0^*$. Then
by the lower semicontinuity of $R(\cdot \lVert\nu)$
$$R(\gamma_0^*\lVert\nu)\leq \liminf_{k\to\infty} R(\gamma_{\delta_k}^*\lVert\nu)\leq \liminf_{k\to\infty}G_{\delta_k\Gamma_0}(\mu,\nu)\leq \lim_{k\to\infty}\delta_k W_{\Gamma_0}(\mu,\nu)=0.$$
Since $R(\gamma_0^*\lVert\nu)\geq 0$ with equality if and only if $\gamma_0^*=\nu$, we conclude $R(\gamma_0^*\lVert\nu)=0$ and $\gamma_0^*=\nu$. Therefore
\begin{align*}
\liminf_{k\to\infty}\frac{1}{\delta_k}G_{\delta_k \Gamma_0}(\mu\lVert\nu)&\geq \liminf_{k\to\infty}\frac{1}{\delta_k}W_{\delta_k\Gamma_0}(\mu,\gamma^*_{\delta_k})\\
&=\liminf_{k\to\infty}W_{\Gamma_0}(\mu,\gamma^*_{\delta_k})\\
&\geq W_{\Gamma_0}(\mu,\gamma^*_0)=W_{\Gamma_0}(\mu,\nu),
\end{align*}
and since the original sequence was arbitrary
$$\liminf_{\delta\to 0}\frac{1}{\delta}G_{\delta\Gamma_0}(\mu\lVert\nu)\geq W_{\Gamma_0}(\mu,\nu).$$

\vspace{\baselineskip}\noindent
2) $W_{\Gamma_0}(\mu,\nu)=\infty.$
If $\liminf_{\delta\to 0}\frac{1}{\delta}G_{\delta\Gamma_0} (\mu\lVert\nu)<\infty$, then there is a subsequence $\left\{\delta_l\right\}_{l\in\mathbb{N}}\subset (0,1)$ that achieves this $\liminf$. 
From essentially the same proof above applied to this subsequence, it can be shown there exists a further subsequence (again denoted $\left\{\delta_{l}\right\}$)  and  $\gamma_0^*\in\mathcal{P}(S)$ such that 
$$G_{\delta_{l}\Gamma_0}(\mu\lVert\nu) =W_{\delta_{l}\Gamma_0}(\mu,\gamma^*_\delta)+ R(\gamma^*_{\delta_{l}}\lVert\nu),$$
and 
$$\gamma^*_{l}\Rightarrow \gamma_0^*.$$
Denote $M\doteq \liminf_{\delta\to 0}\frac{1}{\delta}G_{\delta\Gamma_0} (\mu\lVert\nu)=\lim_{l\to\infty}\frac{1}{\delta_{l}}G_{\delta_{l}\Gamma_0} (\mu\lVert\nu)<\infty$. 
Since for $l$ large enough
$$R(\gamma_{\delta_{l}^*}\lVert\nu)\leq G_{\delta_{l}\Gamma_0}(\mu\lVert\nu)\leq \delta_{l}(M+1),$$
we have 
$$R(\gamma_0^*\lVert\nu)\leq \liminf_{l\to\infty}R(\gamma_{\delta_{l}}^*\lVert\nu)\leq \lim_{l\to\infty}\delta_{l}(M+1)=0,$$
and thus $\gamma_0^*=\nu$. However this leads to 
\begin{align*}
M=\lim_{l\to\infty}\frac{1}{\delta_{l}}G_{\delta_{l}\Gamma_0} (\mu\lVert\nu)&\geq \lim_{l\to\infty}\frac{1}{\delta_{l}}W_{\delta_{l}\Gamma_0} (\mu,\gamma_{\delta_{l}}^*)\\
&=\lim_{l\to\infty}W_{\Gamma_0}(\mu,\gamma_{\delta_{l}}^*)\geq W_{\Gamma_0}(\mu,\nu)=\infty.
\end{align*}
This contradiction implies $$\liminf_{\delta\to 0}\frac{1}{\delta}G_{\delta\Gamma_0} (\mu\lVert\nu)=\infty = W_{\Gamma_0}(\mu,\nu).$$
\end{proof}

\vspace{\baselineskip}
We now consider more refined approximations when $b$ is large.
Previously we described the limiting behavior when we vary the size of $\Gamma$. From Proposition \ref{lim_RE}, we know that when $\mu\not\ll\nu$, $\lim_{b\to\infty} G_{b\Gamma_0}(\mu\lVert\nu) = \infty$. In some applications one might use a large transport cost as ``penalty'' so that while allowing non-absolutely continuous perturbations, control on  $G_\Gamma(\mu\lVert\nu)$ will ensure that $\mu$ is not too far away from $\nu$. 

In the rest of this section, we investigate the behavior when $b\to\infty$, and in particular how $G_{b\Gamma_0}(\mu\lVert\nu)$ will behave for fixed $\mu$ and $\nu$. We only consider the case that $\Gamma_0=\mathrm{Lip}(c,S;C_b(S))$ for some function $c$ satisfies the condition of Theorem \ref{massdual}, Assumption \ref{mea-det} and Assumption \ref{finite}, and $\mu,\nu\in L^1(a)$ with $a$ in Assumption \ref{finite}. We separate the cases depending on whether $\mu$ and $\nu$ are discrete or continuous. The results presented here are only for special cases,
and further development of these sorts of expansions would be useful. 

\subsection{Finitely supported discrete measures}
We will consider the case where $\mathrm{supp}(\nu)$ has finite cardinality, and $\mu$ is also discrete with finite support. 

\begin{theorem}
Suppose $\nu$ and $\mu$ are discrete with finite support, where $\mathrm{supp}(\nu) = \left\{x_i\right\}_{1\leq i\leq N}$ and $\mathrm{supp}(\mu) = \left\{y_j\right\}_{1\leq j\leq M}$. Then there exists $\gamma^*\in\mathcal{P}(S)$ with $\gamma^*\ll\nu$ such that 
\begin{equation}
G_{b\Gamma_0}(\mu\lVert\nu) =bW_{\Gamma_0}(\mu,\gamma^*)+  R(\gamma^*\lVert
\nu)+o(b),
     \label{eqn:exp}
\end{equation}
where $o(b)\leq 0$ satisfies $o(b)\rightarrow 0$ as $b\rightarrow \infty$.
Furthermore, we can characterize $\gamma^*$ as the measure that minimizes $R(\gamma\lVert\nu)$ over the collection of $\gamma\in P(S)$ that satisfy the constraint 
\begin{align}\label{min_gamma}
W_{\Gamma_0}(\mu,\gamma) = \inf_{\theta\ll\nu}W_{\Gamma_0}(\mu,\theta).
\end{align}
If we further assume that 
$$c(y_j,x_i) \neq c(y_j,x_l)$$ 
for $1\leq j\leq M$ and $1\leq i\neq l\leq N$, then $\gamma^*$ has the following form.
Let $S_i$ be the indicies $j$ in $\{1,\ldots,M\}$ for which $x_i$ is the point in $\left\{x_l\right\}_{1\leq l\leq N}$ closest to $y_j$.
Then for $1\leq i\leq N$,
$$\gamma^*(\{x_i\})= \sum_{j\in S_i}\mu(\{y_j\}).$$
\end{theorem}
\begin{remark}
In discrete case, is easily checked that the infimum in (\ref{min_gamma}) is achieved. Take a sequence of $\theta_n\ll\nu$ such that 

$$W_{\Gamma_0}(\mu,\theta_n)\leq \inf_{\theta\ll\nu}W_{\Gamma_0}(\mu,\theta) + 1/n.$$
Since $\theta_n$ is supported on the compact set $\mathrm{supp}(\nu)=\left\{x_i\right\}_{1\leq i\leq N}$
$\left\{\theta_n\right\}_{n\in\mathbb{N}}$ is compact, and hence there exist $\theta^*\ll\nu$ 
and a subsequence $\left\{\theta_{n_k}\right\}_{k\in\mathbb{N}}$
that converges to $\theta^*$ weakly.
By the lower semicontinuity of $W_{\Gamma_0}$
$$W_{\Gamma_0}(\mu,\theta^*) \leq \liminf_{n\to\infty}W_{\Gamma_0}(\mu,\theta_n) \leq \inf_{\theta\ll\nu}W_{\Gamma_0}(\mu,\theta),$$
and therefore $\theta^*$ achieves the infimum of (\ref{min_gamma}).
\end{remark}



\begin{proof}
We use the  representation $G_\Gamma(\mu\lVert\nu)=\inf_{\gamma\in\mathcal{P}(S)}\left\{  R(\gamma\lVert
\nu)+W_{\Gamma}(\mu,\gamma)\right\}$. 
First note that
\begin{align*}
G_{b\Gamma_0}(\mu\lVert\nu) &=\inf_{\gamma\in\mathcal{P}(S)}\left\{  R(\gamma\lVert
\nu)+W_{b\Gamma_0}(\mu,\gamma)\right\}\\
&\leq R(\gamma^*\lVert\nu)+ W_{b\Gamma_0}(\mu,\gamma^*)\\
&=R(\gamma^*\lVert\nu)+ bW_{\Gamma_0}(\mu,\gamma^*).
\end{align*}
Next, fix any $\varepsilon>0$, and take a near optimizer $\gamma_b$, so that for each $b$ 
$$G_{b\Gamma_0}(\mu\lVert\nu)\geq R(\gamma_b\lVert\nu) + W_{b\Gamma_0}(\mu,\gamma_b)-\varepsilon.$$
We must have $\gamma_b\ll\nu$. By (\ref{min_gamma}), we know 
$$W_{b\Gamma_0}(\mu,\gamma_b)= b W_{\Gamma_0}(\mu,\gamma_b) \geq b W_{\Gamma_0}(\mu,\gamma^*)= W_{b\Gamma_0}(\mu,\gamma^*).$$
Thus
\begin{align}
R(\gamma^*\Vert\nu)+W_{b\Gamma_0}(\mu,\gamma^*)&\geq \inf_{\gamma\in P(S)}\left\{  R(\gamma\lVert
\nu)+W_{b\Gamma_0}(\mu,\gamma)\right\} \nonumber\\
& = G_{b\Gamma_0}(\mu,\nu)\nonumber\\
&\geq R(\gamma_b\lVert\nu) + W_{b\Gamma_0}(\mu,\gamma_b)-\varepsilon\nonumber\\
&\geq R(\gamma_b\lVert\nu) + W_{b\Gamma_0}(\mu,\gamma^*)-\varepsilon.\label{eqn:ble}
\end{align}

Since $W_{b\Gamma_0}(\mu,\gamma^*)$ is finite we can subtract it on both sides, and get
$$R(\gamma_b\lVert\nu)\leq R(\gamma^*\lVert\nu)+\varepsilon$$
for any $b<\infty $. Then by \cite[Lemma 1.4.3(c)]{dupell4} $\left\{\gamma_b\right\}_{b\in(0,\infty)}$ is tight. Take a convergent subsequence $\left\{\gamma_{b_k}\right\}$, and denote its limit by $\gamma_{\infty}$. It is easily checked that $\gamma_{\infty}\ll\nu$, so $W_{\Gamma_0}(\mu,\gamma_\infty) \geq W_{\Gamma_0}(\mu,\gamma^*)$. On the other hand, by \eqref{eqn:ble}
\begin{align*}
W_{\Gamma_0}(\mu,\gamma_\infty)- W_{\Gamma_0}(\mu,\gamma^*)&\leq \liminf_{k\to\infty} W_{\Gamma_0}(\mu,\gamma_{b_k})- W_{\Gamma_0}(\mu,\gamma^*)\\
&=\liminf_{k\to\infty}\frac{1}{b_k}(W_{b_k\Gamma_0}(\mu,\gamma_{b_k})-W_{b_k\Gamma_0}(\mu,\gamma^*))\\
&\leq \liminf_{k\to\infty}\frac{1}{b_k}(R(\gamma^*\lVert\nu)-R(\gamma_{b_k}\lVert\nu)+\varepsilon)\\
&\leq \liminf_{k\to\infty}\frac{1}{b_k}(R(\gamma^*\lVert\nu)+\varepsilon)\\
& = 0.
\end{align*}

Thus we conclude that $W_{\Gamma_0}(\mu,\gamma_\infty)= W_{\Gamma_0}(\mu,\gamma^*)$. By the definition of $\gamma^*$  we must have $R(\gamma_\infty\lVert\nu)\geq R(\gamma^*\lVert\nu)$. Choose $k_0$ such that $b_{k_0}\geq 1$.
Then
\begin{align*}
&\liminf_{k\to\infty} \left(G_{b_k\Gamma_0}(\mu\lVert\nu)-[R(\gamma^*\lVert\nu)+b_k W_{\Gamma_0}(\mu,\gamma^*)]\right)\\
&\quad\geq \liminf_{k\to\infty}\left(R(\gamma_{b_k}\lVert\nu)+b_kW_{\Gamma_0}(\mu,\gamma_{b_k})-\varepsilon - (R(\gamma^*\lVert\nu) + b_k W_{\Gamma_0}(\mu,\gamma^*))\right)\\
&\quad\geq \liminf_{k\to\infty} (R(\gamma_{b_k}\lVert\nu)-R(\gamma^*\lVert\nu)) +\liminf_{k\to\infty}b_k(W_{\Gamma_0}(\mu,\gamma_{b_k})-W_{\Gamma_0}(\mu,\gamma^*))-\varepsilon\\
&\quad\geq (R(\gamma_{\infty}\lVert\nu)-R(\gamma^*\lVert\nu)) + \liminf_{k\to\infty}(W_{\Gamma_0}(\mu,\gamma_{b_k})-W_{\Gamma_0}(\mu,\gamma^*))-\varepsilon\\
&\quad\geq 0 + (W_{\Gamma_0}(\mu,\gamma_\infty) - W_{\Gamma_0}(\mu,\gamma^*))-\varepsilon\\
&\quad\geq -\varepsilon
\end{align*}
where the fourth inequality is because $R(\gamma_\infty\lVert\nu) \geq R(\gamma^*\lVert\nu)$ and the lower semi-continuity of $W_{\Gamma_0}(\mu,\cdot)$. Since $\varepsilon>0$ is arbitrary, this establishes \eqref{eqn:exp} along the given subsequence. For any other sequence $\{b_k\}_{k\in\mathbb{N}}$ along which $\lim_{k\to\infty}\left(G_{b_k\Gamma_0}(\mu\lVert\nu)-[R(\gamma^*\lVert\nu)+b_k W_{\Gamma_0}(\mu,\gamma^*)]\right)$ has a limit, we can also take a subsequence from it according to the discussion above. Thus the statement is proved.

The proof of the claimed form for $\gamma^*$ is straightforward and omitted.
\end{proof}

\subsection{An example with $\nu$ is continuous}
To illustrate an interesting scaling phenomenon, here we consider the example with 
$S = \mathbb{R}$, $c(x,y)=|x-y|$,  $\nu=\mbox{Unif}([0,1])$, $\mu=\delta_0$.
Consider $\gamma^*(dx) = c_0 e^{-bx}dx$ and $g^*(x) = -bx$ for $0\leq x\leq 1$, where $c_0$ is the normalizing constant. 
For this example $\Gamma_0=\mathrm{Lip}(c,S;C_b(S))$ is the set of bounded functions over $\mathbb{R}$ with Lipschitz constant 1.  It is easily checked using Theorem \ref{verif} that $\gamma^*$ and $g^*$ are the optimizers in
$$G_{b\Gamma_0}(\mu\lVert\nu)=\inf_{\gamma\in\mathcal{P}(S)}\left\{ 
W_{b\Gamma_0}(\mu,\gamma)+R(\gamma\lVert\nu)\right\}
=\sup_{g\in b\Gamma_0} \left\{\int_S g d\mu-\log\int_S e^g\nu\right\}.$$
Thus we have 
$$G_{b\Gamma_0}(\mu\lVert\nu)=-\int_0^1 bx d\mu - \log\int_0^1 e^{-bx}d\nu =-\log\int_0^1 e^{-bx} dx= \log\left(\frac{b}{1-e^{-b}}\right),$$
and in this case, $G_{b\Gamma_0}(\mu\lVert\nu)$ scales as $\log(b)+o(\log(b))$.

For comparison we consider the optimal transport cost between $\mu$ and $\nu$.
We have
\begin{align*}
W_{bc}(\mu,\nu) &\doteq \sup_{g\in b\Gamma_0}\left\{\int_S g d\mu -\int_S gd\nu\right\} \\
&= b\sup_{g\in \Gamma_0}\left\{\int_S g d\mu -\int_S gd\nu\right\}=bW_c(\mu,\nu) 
\end{align*}
and one can calculate that $W_{c}(\mu,\nu)=1/2.$ Thus $W_{bc}(\mu,\nu)=b/2$, and so $G_{b\Gamma_0}(\mu\lVert\nu)$ gives a much smaller divergence between non absolutely continuous measures $\mu$ and $\nu$ than the corresponding optimal transport cost when the admissible $\Gamma = b\Gamma_0$ is becoming large.

\section{Application to Uncertainty Bounds}

\subsection{Extension to unbounded functions}

From
\[
G_{\Gamma}(\mu\lVert\nu)\doteq\sup_{g\in\Gamma}\left\{  \int_{S}gd\mu-\log
\int_{S}e^{g}d\nu\right\}
\]
we get for all $g\in\Gamma$,
\[
\int_{S}gd\mu\leq G_{\Gamma}(\mu\lVert\nu)+\log\int_{S}e^{g}d\nu.
\]

The inequality above with relative entropy in place of $G_\Gamma(\mu\lVert\nu)$ is the key to uncertainty bounds in \cite{dupkatpanple}. We would like to extend this inequality to unbounded functions.  Define 
$$\hat{\Gamma}_+=\left\{ f:\text{ there exist }g_{i}\in\Gamma\text{ with }c\leq g_{i}%
(x)\uparrow f(x)\text{ for }x\in S\right\},$$
and 
$$\hat{\Gamma}_-=\left\{ f:\text{ there exist }g_{i}\in\Gamma\text{ with }c\geq g_{i}%
(x)\downarrow f(x)\text{ for }x\in S\right\}.$$

\begin{proposition}
For $g\in\hat{\Gamma}_+\cup \hat{\Gamma}_-$, we have 
\begin{equation}
\label{eqn:bineq}
\int_{S}gd\mu\leq G_{\Gamma}(\mu\lVert\nu)+\log\int_{S}e^{g}d\nu. 
\end{equation}

\end{proposition}
\begin{proof}
The proof is straightforward. Take $g\in\hat{\Gamma}_+$. Then there exist $g_i\in\Gamma$, which are bounded below, and increase to $g$ pointwise in $S$. By monotone convergence theorem, 
$$\lim_{i\to\infty}\int_S g_i d\mu=\int_S g d\mu,$$
and 
$$\lim_{i\to\infty}\int_S e^{g_i}d\nu = \int_S e^g d\nu.$$
Since $g_i\in\Gamma$, for all $i$
$$\int_S g_i d\mu \leq G_\Gamma(\mu\lVert\nu) + \log\int_S e^{g_i}d\nu.$$
Taking $i\to\infty$ in the last display gives \eqref{eqn:bineq}. 
For $g\in \hat{\Gamma}_-$ the reasoning is essentially the same.
\end{proof}

\vspace{\baselineskip}
In the case when $\Gamma=\mathrm{Lip}(c,S;C_b(S))$, where $c$ satisfies the conditions introduced in Section \ref{OT}, we can get a stronger version of the result.
The proof is essentially the same as in Lemma \ref{beyond_bounded}, and is omitted.

\begin{proposition}
Assume $c:S\times S\to \mathbb{R}\cup\{+\infty\}$ satisfies Conditions  \ref{con:crep}, \ref{mea-det}, \ref{finite} and \ref{Cond:c_cont}. Fix $\mu,\nu\in L^{1}(a)$. Then for $g\in\mathrm{Lip}(c,S)$
\[
\int_{S}gd\mu\leq G_{\Gamma}(\mu\lVert\nu)+\log\int_{S}e^{g}d\nu.
\]
\end{proposition}








\subsection{Decomposition and scaling properties}

A property of great importance in applications of relative entropy is the chain rule.
When probability measures can be decomposed,
such as when Markov measures on a path space are written as the repeated
integration with respect to transition kernels,
the chain rule allows one to decompose the relative entropy of two such measures on path space in terms of the simpler relative entropies of the transition kernels.
This decomposition also exhibits important scaling properties of relative entropy, e.g.,
that for such Markov measures on path space the relative entropy scales proportionate to the number of time steps.

Except in special circumstances, optimal transport metrics do not possess a property like the chain rule, 
and it is therefore not to be expected that $\Gamma$-divergence would either. 
However,
if one considers certain classes of functions on path space, then one can show there are analogous decomposition and scaling properties. 
In this section we will discuss a setting relevant to many applications,
though the results have many analogues and possible generalizations.

As usual, we assume that $S$ is a Polish space, and let $p:S \times \mathcal{B}(S)$ be a probability transition kernel:
\begin{itemize}
    \item for every $A \in \mathcal{B}(S)$ the map $x \rightarrow p(x,A)$ is Borel measurable, and
    \item for every $x \in S$, $p(x,\cdot)$ is in $\mathcal{P}(S)$.
\end{itemize}
The quantities of interest are large and infinite time averages,
both with respect to time and the underlying distribution,
and we wish to bound in a tight fashion the error in such quantities due to model misspecification.
Thus if $q$ is some other transition kernel, then we seek useful bounds on differences of the form 
\[
\frac{1}{cT}\log E^{\gamma,p} \left[ e^{c\sum_{i=1}^T f({X}_i)} \right]-E^{\theta,q} \left[ \frac{1}{T}\sum_{i=1}^T f({X}_i) \right] ,
\]
where $E^{\gamma,p}$ indicates that the chain uses transition kernel $p$ and initial distribution $\gamma$, and similarly for $E^{\theta,q}$.
Under conditions, relative entropy can provide useful bounds when $q(x,\cdot) \ll p(x,\cdot)$ for a suitable set of $x \in S$.
One question then is under what conditions will the $\Gamma$-divergence allow one to weaken the absolute continuity restriction.
It is also worth noting that even when $q(x,\cdot) \ll p(x,\cdot)$ the bounds obtained using the $\Gamma$-divergence (when applicable) are tighter,
since it is never greater than relative entropy,
and in some cases the improvement can be dramatic. 
These issues will be explored in greater detail elsewhere.

It follows directly from discussion in earlier sections that even in the setting of product measures that one must restrict the class of functions $f$ under consideration. 
When considering Markov measures, the following definition is relevant.
\begin{definition}\label{f_allowed}
For a transition kernel $p$, let 
\[
\mathcal{R}(\Gamma,p) = \left\{ -\log \int_S e^{-g(y)}p(x,dy)-g(x)+a: g \in \Gamma\mbox{ and }a \in \mathbb{R}\right\}.
\]
\end{definition}
Then $\mathcal{R}(\Gamma,p)$ will determine the set of costs $f$ such that bounds can be obtained using the $\Gamma$-divergence.
In particular, we have the following.

\begin{theorem}
\label{thm:ergcost}
Suppose that $f \in \mathcal{R}(\Gamma,p)$ for some $g$ and $a$. Consider any transition kernel $q$ on $S$ and any stationary probability measure $\pi_q$ of $q$.
Then
\begin{align*}
\int_S f(x)\pi_q(dx)&\leq \int_S G_{\Gamma}(q(x,\cdot)\lVert p(x,\cdot)) \pi_q(dx) + a.
\end{align*}
\end{theorem}

\begin{remark}
If $p$ is ergodic then we recognize
\[
 f(x)=-\log \int_S e^{-g(y)}p(x,dy)-g(x)+a
\]
as the equation that uniquely characterizes the multiplicative cost 
\[
a=\lim_{N\rightarrow \infty}\frac{1}{N} \log E^pe^{-\sum_{i=0}^{N-1}f(X_i)},
\]
with $g$ a type of cost potential.
Note that for a given $f$ the function $g$ plays no role in the bound.
We need to check that $f$ is in the range of $\Gamma$ (which of course imposes restrictions on $f$),
but the bound does not depend on knowing the specific form of $g$.
\end{remark}

\begin{proof}
Since $g \in \Gamma$
\begin{align*}
g(x) &= -f(x) - \log\int_S e^{-g(y)}p(x,dy) + a\\
&= -f(x) + \inf_{q(x,dy)}\left[G_\Gamma (q(x,\cdot)\lVert p(x,\cdot)) + \int_S g(y)q(x,dy)\right]+a.
\end{align*}
For the given transition kernel $q$
$$g(x)\leq -f(x)+ \left[G_\Gamma (q(x,\cdot)\lVert p(x,\cdot)) + \int_S g(y)q(x,dy)\right]+a,$$
and integrating both sides with respect to $\pi_q(dx)$ and using $\int_S q(x,dy)\pi_q(dx)=\pi_q(dy)$ gives the result.
\end{proof}

\vspace{\baselineskip}
We next consider two examples to illustrate Definition $\ref{f_allowed}$.
\begin{example}
$S=\mathbb{R}$, $p(x,\cdot)\sim N(\alpha x,\sigma^2)$ is normal distribution with mean $\alpha x$ and variance $\sigma^2$, where $0<\alpha<1$. Let $g(x) = -bx^2 -cx-d$, for $b,c,d \in \mathbb{R}$.

Then direct computation gives that when $1-2b\sigma^2 >0$
\begin{align*}
& -\log\int_S e^{-g(y)}p(x,dy) - g(x) +a \\
&\quad=-\frac{b\alpha^2 x^2+c\alpha x + c^2\sigma^2/2}{1-2b\sigma^2}+bx^2+cx + a\\
&\quad = b\left(1- \frac{\alpha^2}{1-2b\sigma^2}\right)x^2 + c\left(1-\frac{\alpha}{1-2b\sigma^2}\right)x +a.
\end{align*}
Letting $k(b) = b (1- \frac{\alpha^2}{1-2b\sigma^2})$,
$$k'(b) = 1- \frac{\alpha^2}{(1-2b\sigma^2)^2}.$$

Since $1-2b\sigma^2 >0$, we can conclude $k$ reaches its maximum at $1-2b\sigma^2 = \alpha$, i.e., $b= \frac{1-\alpha}{2\sigma^2}$, where $k(b) = \frac{(1-\alpha)^2}{2\sigma^2}$.If $b\to \frac{1}{2\sigma^2}$ then $k(b)\to -\infty$. Also notice that when $b\neq \frac{1-\alpha}{2\sigma^2}$, we can pick $c$ to make the coefficient of $x$ to be any given number. Thus with $p(x,\cdot)\sim N(\alpha x, \sigma^2)$ and $\Gamma  =\left\{bx^2+cx+d:b,c,d\in\mathbb{R}\right\}$,  
$$R(\Gamma,p) = \left\{bx^2+cx+d: b<\frac{(1-\alpha)^2}{2\sigma^2},c,d\in\mathbb{R}\right\}\cup \left\{\frac{(1-\alpha)^2}{2\sigma^2}x^2 +d:d\in\mathbb{R}\right\}.$$
\end{example}

\begin{example}
$S=\left\{  x_{1},x_{2},\dots,x_{n}\right\}  $ is a finite space, and there is
a cost function $c:S\times S\rightarrow\mathbb{R}_{+}$ associated with this
space. Take $\Gamma=\mathrm{Lip}(c,S;C_{b}(S))$. Since $p$ is a transition
matrix we denote $p_{ij}=p(x_{i},x_{j})$ and $P=(p_{ij})_{1\leq i,j\leq n}%
\in\mathbb{R}^{n\times n}$.

A question we ask here is whether there exists $\sigma>0$ such that
$\sigma\Gamma\in R(\Gamma,p)$. In other words, does there exist $\sigma>0$
such that for any $f\in\sigma\Gamma$ we can find $g\in\Gamma$ and
$a\in\mathbb{R}$ such that%
\[
f(x_{i})=-g(x_{i})-\log\sum_{j=1}^{n}p(x_{i},x_{j})e^{-g(x_{j})}+a,\quad
i=1,2,\dots,n.
\]
If $R(\Gamma,p)$ includes such a neighborhood of zero, then when combined with
Theorem \ref{thm:ergcost} it would allow for sensitivity bounds, i.e., bounds on quantities of the
form
\[
\frac{d}{d\theta}\sum_{x\in S}\pi(\theta,x)f(x),
\]
where $f\in\Gamma$, $\pi(\theta,\cdot)$ is the stationary distribution of
$P(\theta)$, $P(0)=P$, and $P(\theta)$ depends smoothly on a vector of
parameters $\theta$ (see \cite{dupkatpanple}). In contrast with \cite{dupkatpanple}, we would not need that
the transition matrices be mutually absolutely continuous.

Since $S$ is finite we write $f_{i}$ for $f(x_{i})$ and let $f=(f_{1}%
,\ldots,f_{n})$, and similarly for $g$. Then the relation above defines a
mapping from $(g,a)$ to $f$, which we denote it by $f=\varphi(g,a)$. Note that%
\[
(0,0,\dots,0)=\varphi((0,0,\dots,0),0),
\]
The $(n,n+1)$ dimensional matrix of partial derivatives takes the form
\[
J=\left[  \left(  P-I\right)  ,\boldsymbol{1}\right]  ,
\]
where $I$ is the $n\times n$ identity matrix and $\boldsymbol{1}$ is a column
vector of ones. If we can show that $J$ is of full rank then the range of the
mapping defined by $J$, i.e., the linearization of $\varphi$ will be onto
$\mathbb{R}^{n}$. Then by the implicit function theorem there will be an open
neighborhood $U$ of $\mathbf{0}\in\mathbb{R}^{n}$ and a continuous function
$\gamma:\mathbb{R}^{n}\rightarrow\mathbb{R}^{n}$ such that for all $f\in U$,%
\[
f=\varphi(0,\gamma(f)).
\]
Since $O\doteq\left\{  (y_{1},y_{2},\dots,y_{n})|(0,y_{1},\dots,y_{n-1}%
)\in\mathrm{int}(\mathrm{Lip}(c,S)),y_{n}\in\mathbb{R}\right\}  $ is open,
$\mathbf{0}\in U\cap\gamma^{-1}(O)\subset\mathbb{R}^{n}$ is also open. Thus we
can pick $\sigma>0$ such that $\mathbf{0}\in \sigma\Gamma\subset U\cap\phi^{-1}(O)$. So we have shown the existence of $\sigma>0$ such that
$\sigma\Gamma\in R(\Gamma,p)$.
\end{example}

Whether or not $J$ is of full rank will depend on the structure of $P$. We
have the following lemma.

\begin{lemma}
Suppose that $S=\bar{S}\cup M$, where $M$ consists of the transient states, and
that when restricted to $\bar{S}$, $P$ is ergodic. Then $J$ is of full rank.
\end{lemma}

\begin{proof}
Let $\pi$ denote the stationary distribution of $P$. Then interpreting $\pi$
as a column vector, it is the unique vector in the null space of $(P-I)^{T}$.
According to the Fredholm alternative, the range of $(P-I)$ is the $n-1$
dimensional collection of vectors $b\in\mathbb{R}^{n}$ such that $\left\langle
b,\pi\right\rangle =0$. Now $\left\langle \boldsymbol{1},\pi\right\rangle >0$,
which shows that $\boldsymbol{1}$ is \textit{not} in the range of $(P-I)$.
Therefore the range of $J$ is all of $\mathbb{R}^{n}$.
\end{proof}

\vspace{\baselineskip}
To give a simple example of how the $\Gamma$-divergence could be used for
model simplification, consider the situation where we are given an ergodic
chain $P$ with state space $\bar{S}$, and would like to replace $P$ by a chain
$Q$ with state space $S=\bar{S}\cup M$, where the new states are intended to
replace a (possibly large) number of states in $\bar{S}$, with the goal being
to maintain good approximation of certain functionals of the stationary
distribution. If $\pi_{q}$ denotes the stationary distribution of $Q$ on $S$
and $\pi_{p}$ that of $P$ on $\bar{S}$, then one could not use relative
entropy to obtain any bounds. Suppose we were to extend $P$ to $\bar{S}\cup M$
(while keeping $P$ as the transition matrix), by making all states in $M$
transient. Then one could use the  $\Gamma$-divergence as long as the
functionals of interest are in $R(\Gamma,p)$ (with respect to the extended
transition probabilities). Note that the location of the new states would be
relevant to this question, since the costs $f$ depend on these locations.
Similarly, one could do sensitivity bounds for non-absolutely continuous
transitions by using such a device.

\section{Conclusion}
In this paper, we defined a new divergence by starting with a variational representation for relative entropy and placing additional restrictions on the collection of test functions used in the representation,
so as to relax the requirement of absolute continuity.
Basic qualitative properties of the divergence were investigated, as well as its relationship with optimal transport metrics. 
Future work will use the divergence to develop uncertainty quantification bounds, sensitivity bounds and methods for model approximation and simplification for stochastic for models without the  absolute continuity requirement.
Also needed is further investigation of qualitative and computational aspects of the $\Gamma$-divergence.

\bibliographystyle{plain}
\bibliography{main}

\vspace{\baselineskip}

\textsc{\noindent P. Dupuis\newline Division of Applied Mathematics\newline
Brown University\newline Providence, RI 02912, USA\newline email:
Paul\_Dupuis@brown.edu \vspace{\baselineskip} }

\textsc{\noindent Y. Mao\newline Department of Mathematics\newline
Harvard University\newline
Cambridge, MA 02138, USA\newline
email: mao@math.harvard.edu}

\end{document}